\documentclass[a4paper]{amsart}
\usepackage{verbatim}
\usepackage[bookmarks, colorlinks=true, pdfstartview=FitV, linkcolor=blue, citecolor=blue, urlcolor=blue]{hyperref}
\usepackage[arrow, matrix, curve]{xy}
\usepackage{enumerate}

\newtheorem{thm}{Theorem}
\newtheorem{Satz}[thm]{Theorem}
\newtheorem{Proposition}[thm]{Proposition}
\newtheorem{Lemma}[thm]{Lemma}
\newtheorem{cor}[thm]{Corollary}

\newtheorem{rem}[thm]{Remark}
\newtheorem{Definition}[thm]{Definition}
\numberwithin{equation}{section} \numberwithin{thm}{section}

\newcommand{\cal}{\mathcal}
\newcommand{\C}{\mathbb{C}} % komplexe
\newcommand{\R}{\mathbb{R}} % reelle
\newcommand{\N}{\mathbb{N}} % natuerliche
\newcommand{\ol}{\overline}

\newcommand{\Ha}{\mathcal{H}}
\newcommand{\bd}{\partial}

\newcommand{\dom}{\mathrm{dom}}

\newcommand{\eps}{\varepsilon}

\newcommand{\sS}{\mathcal{S}}
\newcommand{\sM}{\mathcal{M}}
\newcommand{\sG}{\mathcal{G}}
\newcommand{\sR}{\mathcal{R}}
\newcommand{\sSR}{\mathcal{SR}}
\newcommand{\usS}{\overline{\mathcal S}}
\newcommand{\usM}{\overline{\mathcal M}}
\newcommand{\lsS}{\underline{\mathcal S}}
\newcommand{\lsM}{\underline{\mathcal M}}
\newcommand{\sL}{\mathcal{L}}
\newcommand{\sF}{\mathcal{F}}

\begin{document}
\title[Regular variation, Generalized contents, and fractal strings]{Regularly Varying Functions, Generalized contents, and the spectrum of fractal strings}
\author{Tobias Eichinger}
\address{Technische Universit\"at Berlin, Service-centric Networking,
    Telekom Innovation Laboratories, Ernst-Reuter-Platz 7,
    10587 Berlin, Germany}
\author{Steffen Winter}
\address{Karlsruhe Institute of Technology, Institute of Stochastics, Englerstr.2, 76131 Karlsruhe, Germany}

%\date{\today}
\begin{abstract}
We revisit the problem of characterizing the eigenvalue distribution of the Dirichlet-Laplacian on bounded open sets $\Omega\subset\R$ %$\subset\R$
 with fractal boundaries. It is well-known from the results of Lapidus and Pomerance \cite{LapPo1} that the asymptotic second term of the eigenvalue counting function can be described in terms of the Minkowski content of the boundary of $\Omega$ provided it exists.
He and Lapidus \cite{HeLap2} discussed a remarkable extension of this characterization to sets $\Omega$ with boundaries that are not necessarily Minkowski measurable. They employed so-called generalized Minkowski contents given in terms of gauge functions more general than the usual power functions. %$r^{d-s}$. used in the definition of ordinary Minkowski contents.
The class of valid gauge functions in their theory is characterized by some technical conditions, %which are employed in the course of the proofs of their result. The
the geometric meaning and necessity of which is not obvious. Therefore, it is not completely clear how general the approach is and which sets $\Omega$ are covered. %by the results in \cite{HeLap2}.
Here we revisit these results and put them in the context of regularly varying functions. Using Karamata theory, it is possible to get rid of most of the technical conditions and simplify the proofs given by He and Lapidus, revealing thus even more of the beauty of their results. Further simplifications arise from characterization results for Minkowski contents obtained in \cite{RW13}. We hope our new point of view on these spectral problems will initiate some further investigations of this beautiful theory.
\end{abstract}

\maketitle

\section{Introduction}

Given a bounded open set $\Omega\subset\R^d$ with boundary $F:=\bd \Omega$, consider the eigenvalue problem
\begin{equation}\label{eq:167}
%(P)\text{ : }
\begin{array}{rll} -\Delta u &= \lambda u  &\text{ in } \Omega,\\
u &= 0 &\text{ on } F, \end{array}
\end{equation}
where $\Delta = \sum_{i=1}^d\partial^2/\partial x_i^2$ denotes the Laplace operator.
Recall that $\lambda\in\R$ is called an eigenvalue of \eqref{eq:167}, if there exists a function $u \neq 0$ in $H_0^1(\Omega)$ (the closure of $C_0^\infty(\Omega)$, the space of smooth functions with compact support contained in $\Omega$, in the Sobolev space $H_1(\Omega)$) satisfying $-\Delta u = \lambda u$ in the distributional sense.

It is well-known that the spectrum of $\Delta$ is positive and discrete, i.e.\ the eigenvalues of \eqref{eq:167} form an increasing sequence $(\lambda_{i})_{i\in\N}$  of strictly positive numbers with $\lambda_i\to\infty$, as $i\to\infty$.
If $\Omega$ is interpreted as a vibrating membrane held fixed along its boundary,  the reciprocals of the eigenvalues can be interpreted as the natural frequencies and the corresponding eigenfunctions as the natural vibrations (overtones) of the system. Much work has been devoted to the question how much geometric information about $\Omega$ can be recovered just from listening to its sound.

The famous \textit{Weyl's law} \cite{Weyl1} describes the growth of the eigenvalue counting function $N$, defined by
\begin{equation*}
N(\lambda) := \#\{ i \in \N :  \lambda_i \leq \lambda    \}, \quad \lambda>0.
\end{equation*}
It states (originally for sufficiently smooth domains, but nowadays known to hold for arbitrary bounded open sets $\Omega\subset\R^d$, see \cite{BiSo,Me}) that $N$ is asymptotically equivalent to the so-called \emph{Weyl term} $\varphi$ 
given by
\begin{equation}\label{eq:169}
\varphi(\lambda) := (2\pi)^{-d}\omega_d|\Omega|_d \lambda^{d/2}.
\end{equation}
Here $\omega_d$ denotes the volume of the $d$-dimensional unit ball and $|\cdot|_d$ is the Lebesgue measure in $\R^d$. % which is basically a power function $aX^{d/2}$ (with constant $a$).
Since the dimension $d$ of $\Omega$ as well as its volume appear in the Weyl term, we are not only able to infer the dimension of $\Omega$ from the growth of $N$ but also its volume. For sets with sufficiently smooth boundaries, Weyl conjectured in \cite{Weyl2} the existence of a second additive term of the order $(d-1)/2$ in the asymptotic expansion of $N$ with a prefactor being determined by the surface area of the boundary $F$ of $\Omega$:
\begin{equation}\label{eq:170}
N(\lambda) = \varphi(\lambda) + c_{d-1} \Ha ^{d-1}(F)\lambda^{(d-1)/2} + o(\lambda^{(d-1)/2}), \text{ as }\lambda \rightarrow \infty,
\end{equation}
with a constant $c_{d-1}$ solely depending on the dimension $d-1$ of $F$. Many people have contributed to the verification and generalization of Weyl's conjecture, culminating in the results of Seeley \cite{Seeley1, Seeley2}, Pham The Lai \cite{Ph} (who established that for all bounded open sets $\Omega$ with $C^\infty$-boundary the order $\frac{d-1}2$ of the second term is correct) and Ivrii \cite{Ivrii}, who proved the correctness of the prefactor under some additional assumption (roughly, that the set of multiply reflected periodic geodesics in $\ol{\Omega}$ is of measure zero). Therefore, for sets with smooth boundaries, one can `hear' not only the volume of the set but also the dimension and measure of its boundary.

And yet, what if the boundary is non-smooth? Motivated by experiments on wave scattering in porous media, the physicist M.~V.~Berry conjectured in \cite{Berry1,Berry2} that for general bounded open sets $\Omega$ the second term in \eqref{eq:170} should be replaced by $c_{n,H}\Ha^{H}(F)\lambda^{H/2}$ involving Hausdorff dimension $H$ and Hausdorff measure $\Ha^H(F)$ of the boundary $F$ of $\Omega$. This was disproved by Brossard and Carmona \cite{BrCa}, who suggested to use instead Minkowski content and dimension in the second term. More precisely, these notions need to considered relative to the set $\Omega$.
Recall that for any bounded set $\Omega\subset\R^d$, $F:=\bd \Omega$ and $s\geq0$, the \emph{$s$-dimensional Minkowski content of $F$} (relative to $\Omega$) is defined  by  %\marginpar{relative Inhalte betrachten?}
\begin{equation}
{\sM}^s(F) = \lim_{r\searrow 0}\frac{|F_r\cap\Omega|_d }{r^{d-s}},
\end{equation}
provided the limit exists. Here 
$F_r:=\{x\in\R^d: d(x,F)\leq r\}$ is the \emph{$r$-parallel set} of $F$. We write $\lsM^s(F)$ and $\usM^s(F)$ for the corresponding lower and upper limits. The numbers $\overline{dim}_M F:=\inf\{s\geq 0: \usM^s(F)=0\}$ and $\underline{dim}_M F:=\inf\{s\geq 0: \lsM^s(F)=0\}$ are the \emph{upper and lower Minkowski dimension} of $F$, and in case these numbers coincide, the common value is known as the Minkowski dimension $\dim_M F$ of $F$ (relative to $\Omega$).
If $\sM^s(F)$ is both positive and finite (for some $s$), we say that $F$ is \emph{($s$-dimensional) Minkowski measurable}, and if $0<\lsM^s(F)\leq\usM^s(F)<\infty$, then $F$ is called \emph{($s$-dimensional) Minkowski nondegenerate}. Note that in this case, necessarily $\dim_M F=s$.
(In the above notation and also in most of the further discussion, we suppress the dependence on $\Omega$, but of course it should be kept in mind that all these notions are considered relative to $\Omega$ here.)

Lapidus \cite{Lap1} established that indeed the (upper) Minkowski dimension $D$ of $F=\bd\Omega$ gives the correct order of growth for the second term (provided $\usM^D(F)<\infty$), and
 formulated the so-called \textit{Modified Weyl-Berry  conjecture} (MWB) under the assumption that $F$ is Minkowski measurable:
\begin{equation} \label{eq:MWB}
N(\lambda) = \varphi(\lambda)-c_{d,D} \sM^D(F) \lambda^{D/2} +o(\lambda^{D/2})\text{, as } \lambda \rightarrow \infty.
\end{equation}
Here $c_{d,D}$ is a constant depending only on $d$ and the Minkowski dimension $D = \dim_M(F)$ of $F$.  Lapidus and Pomerance showed in~\cite{LapPo1} that the MWB holds in dimension $d=1$, and in \cite{LapPo2} they constructed counterexamples, which disprove MWB in any dimension $d\geq 2$. In ~\cite{LapPo1}, they also established the following more general result, which characterizes the second order asymptotics of $N$ for sets $\Omega\subset\R$ with a Minkowski nondegenerate boundary $F$.

Recall that (for some $a\in [0,\infty]$) two positive functions $h_1,\,h_2$ defined in some neighborhood of $a$ %$:(0,\infty)\to(0,\infty)$
are called \emph{asymptotically similar as $y\to a$}, in symbols $h_1(y)\asymp h_2(y)$, as $y\to a$, if and only if there are positive constants $\underline{c},\,\overline{c}$ such that $\underline{c}\, h_1(y) \leq h_2(y) \leq \overline{c}\, h_1(y)$ for all $y$ close enough to $a$. \label{def:asymp}
%We write $h_1(y)\asymp h_2(y)$, as $y\to a$ for this.
Moreover, $h_1$ and $ h_2$ are \emph{asymptotically equal, as $y\to a$},  $h_1(y)\sim h_2(y)$, as $y\to a$, if and only if $\lim_{y\to a} h_1(y)/h_2(y)=1$.

\begin{thm}[{\cite[Theorem 2.4]{LapPo1}}]\label{theo:1.1}
Let $\Omega\subset\R$ be a bounded open set and $F:=\bd \Omega$. Let $ D \in (0,1)$ and %$($i.e. in particular with $|F|_d=0)$,
let ${\sL} =(l_j)_{j\in \N}$ denote the associated fractal string. Then the following assertions are equivalent:
	\begin{enumerate}
		\item [(i)] $0 < \lsM^D(F) \leq \usM^D(F) <\infty $,
		\item [(ii)] $l_j\asymp j^{-\frac{1}{D}}$, as $j \rightarrow \infty$,
		\item [(iii)] $\sum_{j=1}^{\infty} \{l_jx\} \asymp x^D$, as $x \rightarrow\infty$,
		\item [(iv)] $\varphi(\lambda) - N(\lambda) \asymp \lambda^{\frac{D}{2}}$, as $\lambda \rightarrow \infty$.
	\end{enumerate}
\end{thm}
Here and throughout, $\{y\} := y- [y]$ denotes the fractional part and $[y]$ the integer part of a number $y \in \R$. (The connection between (iii) and (iv)  is due to the general relation \eqref{eq:N-delta}, see also Remark~\ref{rem:N-delta}.)
%In the sequel, we will abbreviate it by $\delta(x)$, i.e.\ we set
%\begin{align}
%  \label{eq:delta-def}
%  \delta(x):=\sum_{j=1}^{\infty} \{l_jx\},\, x>0.
%  \end{align}

%[Evtl. an dieser Stelle zunächst Minkowski-messbaren Fall diskutieren.!!!!]
Given a bounded open set $\Omega\subset\R^d$, unfortunately, very often its boundary $F=\bd\Omega$ is neither Minkowski measurable nor Minkowski nondegenerate. Its volume function $ V_F(r):=\left|F_r\cap \Omega\right|_d$, $r>0$ may exhibit a growth behavior which differs significantly from the behavior of the functions $r^{1-D}$, as $r\searrow 0$ captured by the Minkowski contents. Substituting power functions with some well-chosen more general \textit{gauge functions} $h:(0,\infty)\to(0,\infty)$, it may be possible to understand much better the behavior of $V_F$ at the origin. This leads to the notion of generalized Minkowski contents: For any function  $h:(0,\infty)\to(0,\infty)$, the \emph{$h$-Minkowski content} of $F$ (relative to $\Omega$), is defined by
\begin{equation} \label{eq:h-Mink-def}
{\sM}(h;F) := \lim_{r\searrow 0 }\frac{|F_r\cap\Omega|_d}{h(r) },
\end{equation}
provided the limit exists. We denote by $\lsM(h;F)$ and $\usM(h;F)$ the corresponding lower and upper limits. Moreover, we call $F$ \emph{$h$-Minkowski measurable}, if and only if $0<{\sM}(h;F)<\infty$, and \emph{$h$-Minkowski nondegenerate}, if $0<\lsM(h;F)\leq\usM(h;F)<\infty$. (Note that again all these notions are relative to the set $\Omega$.)

Given some $\Omega\subset \R^d$ with boundary $F$ and some $h$ such that the $h$-Minkowski content $\sM(h;F)$ exists, the natural question arises, whether the second order asymptotic term of the eigenvalue counting function $N$ is still determined by the asymptotic behavior of the volume function  $V_F$, now described by $\sM(h;F)$.

A partial answer to this question has been given by He and Lapidus \cite{HeLap2}. Introducing a class $\sG_s$ of gauge functions (for each $s\in (0,1)$; cf. Definition \ref{def:gd}), they generalized the results in \cite{LapPo1} employing $h$-Minkowski contents with gauge functions $h\in \sG_s$. In particular, they obtained a generalization of \eqref{eq:MWB} to $h$-Minkowski measurable boundaries as well as the following generalization of Theorem \ref{theo:1.1} to sets with $h$-Minkowski nondegenerate boundaries:

\begin{thm}[{\cite[Theorem 2.7]{HeLap2}}]\label{theo:1.2}
Let $\Omega\subset\R$ be a bounded open set and $F:=\bd \Omega$. %Assume $\dim_MF = D \in (0,1)$ and
%$($i.e., in particular with  $|F|_d=0)$
Denote by ${\mathcal{L}} =(l_j)_{j\in\N}$ the associated fractal string and let $h\in \sG_{1-D}$ for some $D\in(0,1)$. Then the following assertions are equivalent:
	\begin{enumerate}
		\item [(i)] $0 < {\lsM}(h;F) \leq {\usM}(h;F)<\infty $,
		\item [(ii)] $l_j\asymp g(j)$, as $j \rightarrow \infty$,
		\item [(iii)] $\sum_{j=1}^{\infty} \{l_jx\} \asymp f(x)$, as $x \rightarrow \infty$,
		\item [(iv)] $\varphi(\lambda) - N(\lambda) \asymp f(\sqrt{\lambda})$, as $\lambda \rightarrow \infty$,
	\end{enumerate}
	where $g(x) := H^{-1}({1}/{x})$ with $H(x):=x/h(x)$, and   $f(x):=1/{h(1/x)}$.
%$\{y\} = y- [y]$ is the fractional part and $[y]$ the integer part of $y \in \R$.
\end{thm}

The classes $\sG_s$ of gauge functions appearing in Theorem \ref{theo:1.2} are defined as follows.

 \begin{Definition}[{\cite[Definition 2.2]{HeLap2}}]\label{def:gd}
 Let $s\in (0,1)$. A function $h:(0,\infty) \rightarrow (0,\infty)$ is in $ \sG_s$, if and only if the following three conditions are satisfied:
 \begin{enumerate}
 	\item[(H1)] $h$ is a continuous, strictly increasing function with $\lim_{x\searrow 0}h(x)=0$,\\
 $\lim_{x\searrow 0}{h(x)}/{x}  =\infty$ and $\lim_{x\rightarrow \infty}h(x) = \infty$.
 	\item[(H2)] For any $t>0$,
 	\begin{equation*}
 	\lim_{x\searrow 0}\frac{h(tx)}{h(x)} = t^s,
 	\end{equation*}
 	where the convergence is uniform in $t$ on any compact subset of $(0,\infty)$.
 	\item[(H3)] There exist some constants $\tau \in (0,1)$, $ m>0$, $x_0$, $t_0 \in (0,1]$ such that
 	\begin{equation*}
 	\frac{h(tx)}{h(x)} \geq mt^\tau,
 	\end{equation*}
 	for all $0<x\leq x_0 $, $0<t\leq t_0$.
 \end{enumerate}
\end{Definition}

From the definition of the classes $\sG_s$ it is not clear at all, how restrictive they are with respect to the family of sets covered by this approach. Given $\Omega$, it is not easy to decide, whether some suitable gauge function $h$ exists in $\sG_s$ (for some $s$) or not, such that the boundary $F=\bd \Omega$ is $h$-Minkowski nondegenerate or even $h$-Minkowski measurable. It is also not clear whether all of the conditions in the definition of $\sG_s$ are really necessary to establish the derived connections between the eigenvalue counting function $N$ and the volume function $V_F$. On the other hand, there might be a simpler class of gauge functions capable of serving the same purpose.

In the present paper we aim at shedding some light on these problems. We will introduce (for each $\varrho\in\R$) the class $\sSR_\varrho$ of $C^1$-smooth gauge functions $h$ that are regularly varying with index $\varrho$, i.e.\ which satisfy $\lim_{x\searrow 0}\frac{h(tx)}{h(x)} = t^\varrho$ for any $t>0$ (see Definition~\ref{def:reg} and compare with condition (H2) above). Using results from Karamata theory, the theory of regularly varying functions, we will establish results for the classes $\sSR_\varrho$ completely analogous to those obtained in \cite{HeLap2}.
%, e.g.\ an analogue of Theorem~\ref{theo:1.2}.

The new classes $\sSR_\varrho$ of gauge functions have several advantages compared to the classes $\sG_s$ in He and Lapidus \cite{HeLap1, HeLap2}. First, their definition is simpler involving only two conditions - regular variation and smoothness (and the latter will turn out to be a convenience rather than a restriction). Second, the new analogues of Theorem~\ref{theo:1.2} and \eqref{eq:MWB} for $\sSR_\varrho$, stated in Theorem~\ref{theo:main}, hold for all $\Omega\subset\R$ covered by the old result. (We were hoping to extend the family of sets $\Omega$ covered, but, in fact, we will show that exactly the same family of sets is addressed by our results.) %To be more precise (and more specific), given any domain $\Omega\subset \R$ and $h\in \sG_s$ such that $F$ is $h$-Minkowski measurable, there is always an asymptotically equivalent gauge function $\tilde h\in \sSR_s$ such that $\sM(\tilde h,F)=\sM(h,F)$, and vice versa.
Third, using regularly varying gauge functions allows some structural insights, e.g.\ concerning (generalized) Minkowski measurability.
Last but not least, the proofs simplify significantly when using the new class. Employing Karamata theory, allows to separate very clearly the main geometric (and algebraic) arguments from pure `asymptotic calculus'. Moreover, the differentiability allows amongst others to apply some characterization results of Minkowski contents in terms of S-contents from \cite{RW13}.  S-contents describe the asymptotic behavior of the parallel surface areas of the given set $F$. They have been used to simplify some parts of the proof of Theorem~\ref{theo:1.1}, see \cite[\S 5]{RW13}. In analogy to this approach, we add a criterion in terms of generalized S-contents, which will admit further simplification of the proofs also in this setting of general gauge functions.

For these reasons, the new classes $\sSR_\varrho$ are certainly helpful in understanding which sets are covered by the results of He and Lapidus and which are not and whether the theory can be extended further. We see our discussion as a starting point for further investigations, paving the road for a further generalization of the theory.

\section{Regularly varying functions and statement of the main results} \label{sec:results}

%PLAN: Define $\sR$ and $\sSR$ here and state the equivalence of $\sSR_D$ and $\sG_D$. Then state the main result ...
We will now define the classes of gauge functions that we are going to use and state our main results. We start with regular variation, a notion dating back to Jovan Karamata in the 1930'ies. Nowadays, the theory of such functions is known as \emph{Karamata theory}.
While the classic theory is formulated for regular variation at $\infty$, see also Definition~\ref{def:reg2} below, we will require regular variation at $0$.

\begin{Definition}\label{def:reg}
	Let $\varrho\in \R$.
A function $h:(0,y_0)\to (0,\infty)$ (with $y_0>0$, possibly $\infty$) is called \emph{regularly varying (at $0$) of index $\varrho$}, if and only if $h$ is measurable and satisfies for all $t>0$
\begin{equation*}
\lim_{y\searrow 0}\frac{h(ty)}{h(y)} = t^\varrho.
\end{equation*}
We write $\sR_\varrho$ for the family of these functions. % and set $\sR:=\bigcup_{\varrho\in(0,1)} \sR_\varrho$.
Functions $h\in\sR_0$ are also called \emph{slowly varying (at $0$)}.
\end{Definition}

Furthermore, a regularly varying function $h\in \sR_\varrho$ is in the class $\sSR_\varrho$ of \emph{smoothly varying functions (at $0$) of index $\varrho$}, if and only if $h$ is $C^1$-smooth in a neighborhood of $0$, i.e.\
	%\begin{equation*}
	$h \in C^1(0,y_1],$
	%\end{equation*}
	for some $y_1>0$. %Let $\sSR:=\bigcup_{\varrho\in\R}\sSR_\varrho$.
%\end{Definition}

We first discuss the relation of the classes $\sG_\varrho$, $\sR_\varrho$ and $\sSR_\varrho$. It turns out that they are asymptotically equivalent, by which we mean the following: Two classes of functions are \emph{asymptotically equivalent}, if for any function $h$ in one class there is an asymptotically equivalent function $\widetilde h$ in the other class, i.e.\ one such that $\widetilde h(y)\sim h(y)$, as $y\searrow 0$, and vice versa. (It is easy to see that this is an equivalence relation on classes of functions, justifying the notion.)

\begin{thm} \label{theo:classes}
  For any $\varrho\in(0,1)$, the three classes $\sG_\varrho$, $\sR_\varrho$ and $\sS_\varrho$ are  asymptotically equivalent.
\end{thm}

%This result
It is rather obvious that  any $h\in\sG_\varrho$ is in $\sR_\varrho$ (because of condition (H2)), similarly we have $\sSR_\varrho\subset\sR_\varrho$, but in both cases a reverse relation is not obvious. A proof is given in Section~\ref{section:4}, when we discuss regularly varying functions in more detail (see Corollary~\ref{prop:3} and Theorem~\ref{thm:SR=G}). % that any $h\in\sR_\varrho$ has an asymt.
Theorem~\ref{theo:classes} is not particularly deep but very useful. It clarifies that we have some freedom in choosing gauge functions. Regular variation is the basis of all three classes, but we may impose some additional properties for our convenience without loosing generality. This is exemplified by the following simple statement saying that generalized Minkowski contents are invariant with respect to asymptotic equivalence.

\begin{Proposition}\label{lemma:4}
	Let $F\subset \R^d$ be compact and let $h,\tilde h:(0,\infty) \rightarrow (0,\infty)$  be gauge functions such that  $h(y) \sim \tilde{h}(y)$, as $y\searrow 0$. Then $F$ is $h$-Minkowski nondegenerate if and only if $F$ is $\tilde h$-Minkowski nondegenerate. More precisely, one even has
% \begin{equation*}
%	0< \underline{\cal{M}}(h;F) \leq \ol{\cal{M}}(h;F) < \infty.
%	\end{equation*} Then for any asymptotically equivalent function $\tilde{h}$ $($i.e. $h(y) \sim \tilde{h}(y)$, as $y\searrow 0)$ the upper and lower generalized Minkowski contents coincide
	\begin{equation*}
	0< \underline{\cal{M}}(\tilde{h};F) =\underline{\cal{M}}(h;F)<\infty \quad\text{ and }\quad  0< \ol{\cal{M}}(\tilde{h};F) = \ol{\cal{M}}(h;F) < \infty.
	\end{equation*}
%	i.e. $F$ is $\tilde{h}$-Minkowski nondegenerate.\\
	In particular, $F$ is $h$-Minkowski measurable if and only if $F$ is $\tilde{h}$-Minkowski measurable.
	\end{Proposition}

In connection with Proposition~\ref{lemma:4}, Theorem~\ref{theo:classes} clarifies that when searching for a suitable gauge function $h$ (for a given set $F$) within the class of regularly varying functions $\sR_\varrho$ (such that $F$ becomes $h$-Minkowski measurable/nondegenerate), then we can as well restrict to one of the classes $\sG_\varrho$ or $\sSR_\varrho$. Note also that the existence of such an $h\in\sR_\varrho$ necessarily implies that $\dim_M F$ exists and equals $d-\varrho$, cf.\ Remark~\ref{rem:8}.

From the above considerations %regarding the equivalence of the classes $\sG_\varrho$, $\sR_\varrho$ and $\sS_\varrho$ and generalized Minkowski measurability,
one might get the impression that not much can be gained from reproving the results of He and Lapidus, obtained for gauge functions in $\sG_\varrho$, in the context of the classes $\sR_\varrho$ and $\sSR_\varrho$. Indeed, the above results make very clear that not a single additional set $\Omega$ will be covered by the new statements. However, even though our original intention was an extension of the theory to more general sets $\Omega$, our approach using regularly varying functions is still very useful, as it provides a more structural point of view on the studied classes of gauge functions and, more importantly, as it clarifies and simplifies many of the arguments of the long and technical proofs in \cite{HeLap2}. It shows that regular variation is the essential property which makes things work.

As an interesting side result, we mention that generalized Minkowski measurability with respect to regularly varying functions can be characterized without referring to any particular gauge function:

\begin{thm} \label{theo:mink}
  Let $\Omega\subset\R^d$ be a bounded open set and $F=\bd\Omega$. There exists a gauge function $h\in\sR_\varrho$ for $F$ such that $F$ is $h$-Minkowski measurable, if and only if the volume function $V_F$ of $F$ (i.e., the function $r\mapsto |F_r\cap\Omega|_d$, $r>0$) is regularly varying with index $\varrho$, i.e.~if and only if $V_F\in\sR_\varrho$.
\end{thm}

This means that the assumed regular variation of the gauge function reflects an essential property of the volume function itself. In particular, if a set is $h$-Minkowski measurable for some $h\in\sR_\varrho$, then it cannot be $g$-Minkowski measurable for some gauge function $g$ outside the class $\sR_\varrho$ and vice versa.
A proof of Theorem~\ref{theo:mink} is given in Section~\ref{sec:karamata}, see page~\pageref{proof:mink}. %Note also that $V_F\in\sR_\varrho$  means that $\dim_M F$ exists and equals $d-\varrho$, cf.\ Remark~\ref{rem:8}.

Before we formulate our main result, we recall the notion of \emph{(generalized) S-content}.
%Recall the \emph{upper and lower generalized $S$-contents}
For any gauge function $h:(0,y_0)\to(0,\infty)$, the \emph{upper and lower h-S-contents} of $F:=\bd \Omega$ (relative to $\Omega$) are defined by
\begin{align}\label{eq:h-S-def}
	\quad\usS(h;F)&:= \limsup_{r \searrow 0}\frac{{\cal{H}}^{d-1}(\partial F_r\cap \Omega)}{h(r)} \text{ and }
	\lsS(h;F) = \liminf_{r \searrow 0}\frac{{\cal{H}}^{d-1}(\partial F_r\cap\Omega)}{h(r)},
\end{align}
respectively. If $\lsS(h;F)=\usS(h;F)$, then the common value $\sS(h;F)$ is the \emph{$h$-S-content} of $F$.
There are close general relations between (generalized) Minkowski and S-contents of a set $F$. In particular, for differentiable gauge functions $h$ the $h$-Minkowski content $\sM(h;F)$ and the $h'$-S-content $\sS(h';F)$ coincide whenever one of these contents exists. We refer to \cite{RW13,W16} and Section~\ref{sec:proofs1} for more details. Moreover, in dimension 1 there are close relations between the (generalized) S-content of a set and the so-called string counting function of the associated fractal string, see Section~\ref{sec:proofs2} for details.

Another advantage of smooth gauge functions is that,
for any $h\in\sSR_{1-D}$, $D\in(0,1)$, the following two auxiliary functions $f$ and $g$ are well-defined for $x$ large enough (i.e.\ for all $x\in[x_0,\infty)$ for some $x_0\geq 0$):
\begin{align}
   \label{eq:f-and-g} f(x) = x\, h(1/x) \quad \text{ and } \quad g(x) = H^{-1}(1/x),
\end{align}
where $H(x) = x/{h(x)}, x\in\dom(h)$ and $H^{-1}$ is the inverse of $H$, we refer to Proposition~\ref{prop:g-and-f} for details.
$g$ and $f$ take the same role as the corresponding functions in Theorem~\ref{theo:1.2}. They characterize the decay of the lengths of  the associated string $(l_j)$, as $j\to\infty$ and of the \emph{packing defect} $\delta(x):= \sum_{j=1}^{\infty
			}\{l_jx\}$, as $x\to\infty$, respectively.

\begin{Satz}[Main Theorem]\label{theo:main}
Let $\Omega\subset\R$ be a bounded open set and $F:=\bd \Omega$.  Assume $\dim_MF = D \in (0,1)$ and
%$($i.e., in particular with  $|F|_d=0)$
denote by ${\mathcal{L}} =(l_j)_{j\in\N}$ the associated fractal string.
%Let $F\subseteq \R$ be a compact set with Minkowski dimension $dim_MF = D \in (0,1)$ and let ${\sL} =(l_j)_{j\in \N}$ be the associated fractal string.
Let $h \in \sSR_{1-D}$ and let $f$ and $g$ be given as in \eqref{eq:f-and-g}.
	\begin{enumerate}[(1)]		
		\item[I.] $($Two-sided bounds.$)$ The following assertions are equivalent:
		\begin{enumerate}[(i)]
			\item $0< \lsM(h;F) \leq  \usM(h;F) < \infty$,
			\item $0< \lsS(h';F) \leq  \usS(h';F) < \infty$,
			\item $l_j \asymp g(j)$, as $j\rightarrow \infty$,
			\item $\sum_{j=1}^{\infty
			}\{l_jx\} \asymp f(x)$, as $x\rightarrow \infty$,
			\item $\varphi(\lambda) - N(\lambda) \asymp f(\sqrt{\lambda})$, as $\lambda \rightarrow \infty$.
		\end{enumerate}
		\item[II.] $($Minkowski measurability.$)$ The following assertions are equivalent:
		\begin{enumerate}
			\item[(vi)] $F$ is $h$-Minkowski measurable,
			\item[(vii)] $F$ is $h'$-S measurable,
			\item[(viii)] $l_j \sim L\,g(j)$, as $j\rightarrow \infty$ for some positive $L>0$.
		\end{enumerate}
		Under these latter assertions, $h$-Minkowski content, $h'$-S-content of $F$ and the constant $L$ are connected by the relation
		\begin{equation} \label{eq:main-M-S-L}
		\sM(h;F) = \sS(h';F) = \frac{2^{1-D}L^{D}}{1-D}.
		\end{equation}
Moreover, these assertions imply
\begin{align} \label{eq:main-N}
   N(\lambda)=\varphi(\lambda)- c_{1,D} \sM(h; F) f(\sqrt{\lambda}) + o(f(\sqrt{\lambda})), \quad \text{ as } \lambda\to\infty,
\end{align}
where $c_{1,D}=2^{D-1}\pi^{-D}(1-D)(-\zeta(D))$ with $\zeta$ being the Riemann zeta function.
	\end{enumerate}
\end{Satz}

Note that Theorem~\ref{theo:main} comprises the classical results of Lapidus and Pomerance \cite[cf.\ Theorems 2.1-2.4]{LapPo1} for sets $\Omega$ with Minkowski measurable/nondegenerate boundaries as a special case (by choosing the gauge function $h(y)=y^{1-D}$). It parallels the results of He and Lapidus in \cite[Theorems 2.4-2.7]{HeLap2} obtained for gauge functions $h$ of the classes $\sG_{1-D}$.
Our proofs bring the generalized theory of He and Lapidus in some parts closer back to the original arguments in the proofs of \cite{LapPo1} (in some other parts, however, we will use results on S-contents from \cite{RW13}).

\begin{rem}
\emph{One may wonder, whether also a converse to the last assertion in Theorem~\ref{theo:main} holds, i.e.\ whether the existence of an asymptotic second term for $N$ as in \eqref{eq:main-N} (assumed to hold for some $f$ given in terms of some $h$ as in \eqref{eq:f-and-g}) implies the $h$-Minkowski measurability of the boundary $F$. He and Lapidus have addressed this question (for $h\in\sG_{1-D}$). In \cite[Theorem 2.9]{HeLap2} they showed in particular that such a converse fails for all dimensions $D$, for which the Riemann zeta-function has a zero on the line $\Re(s)=D$, generalizing thus earlier results of Lapidus and Maier \cite{LapM1, LapM2} and showing a connection of this question to the Riemann hypothesis. More precisely, the converse fails (for the above mentioned $D$) for any $h\in\sG_{1-D}$ which is differentiable and satisfies $y h'(y)/h(y)\geq \mu$ for some $\mu>0$ (and all $y$). We point out that the differentiablity implies $h\in\sSR_{1-D}$ and that, by Lemma~\ref{lemma:essgauge} below, we have thus $\lim_{y\searrow 0}y h'(y)/h(y)={1-D}$, implying in particular that
$y h'(y)/h(y)\geq (1-D)/2=:\mu$ for all $y$ in a suitable neighborhood of $0$. Hence \cite[Theorem 2.9]{HeLap2} directly applies to any $h\in\sSR_{1-D}$. We are optimistic that our methods might also help to simplify the proof of this statement, but we have not yet tried to do this.}
\end{rem}

\begin{rem}
\emph{In the present paper, we have restricted ourselves to subsets $\Omega$ of $\R$.
In \cite{HeLap2}, also some results in the higher dimensional case are obtained, i.e.\ for bounded open sets $\Omega\in\R^d$, $d>1$. It is known that the MWB conjecture fails in this case and therefore these results are necessarily of a weaker nature. More precisely, under the assumption that the upper $h$-Minkowski content of $F:=\bd \Omega$ is finite (for some suitable gauge function $h$), an upper bound for the second asymptotic term of $N$ is derived in terms of the function $f$ (determined by $h$ as in \eqref{eq:f-and-g}). Although, the assumptions on the gauge function $h$ are slightly different in these higher dimensional results, it may be worth to revisit them in the light of the methods used here.}
\end{rem}

Although we have clarified significantly the (one-dimensional) theory of He and Lapidus by showing that regular variation is the driving force behind these results (while the other conditions can be omitted or follow automatically), we think that this is not the end of the story. Several examples of sets $\Omega$ and suitable associated gauge functions $h$ have been discussed in \cite[cf.\ \S7 and the Appendix]{HeLap2}, to which the results apply. Since all occurring gauge functions $h$ are differentiable and thus in the classes $\sSR_{1-D}$, they can also serve as examples for Theorem~\ref{theo:main} (and to apply the theorem, it is not necessary to check any of the technical conditions (H1)-(H3), only the regular variation). However, there might exist sets $\Omega$, for which regularly varying gauge functions are not suitable but for which nevertheless results similar to Theorem~\ref{theo:main} hold. It would be desirable to have such an example or to prove that it does not exist. In either case, we believe that the concept of regular variation -- and its various generalizations studied in Karamata theory (e.g.\ \emph{de Haan theory}), see e.g.\ the monograph \cite{BGT} -- will be useful for such an attempt.

\begin{rem}
  \emph{The results of Lapidus and Pomerance \cite{LapPo1,LapPo2} and in particular the discovered connections to the Riemann zeta function, have initiated an extensive study of fractal strings, which led to the theory of complex dimensions, see \cite{FGCD} and the many references therein. Nowadays these spectral problems and many related questions, including Minkowski measurability, are studied with the help of various zeta functions, which allow to derive explicit formulas for functions such as the the eigenvalue counting function or the tube volume. We have not made any efforts to relate our results to this theory. It seems, however, that sets that are generalized Minkowski measurable/nondegenerate (but not Minkowski measurable/nondegenerate) are not covered so far by the results in this theory. In the last years, even extensions (of some part of the theory) to higher dimensions are being discussed, see \cite{LRZ}. In the language of \cite[\S6.1]{LRZ}, where a classification scheme for sets $\R^d$ is proposed in terms of the behaviour of their tube functions, the boundaries $F$ studied here and in \cite{HeLap2} fall into the class of \emph{weakly degenerate sets}.}
\end{rem}

The remaining parts of the paper are organized as follows.
After some preliminary considerations concerning asymptotic similarity in Section~\ref{sec:pre}, we recall in Section~\ref{sec:karamata} some results from Karamata theory and discuss various useful consequences. On the way, we will prove the asymptotic equivalence of the classes $\sG_\varrho$, $\sR_\varrho$ and  $\sSR_\varrho$ (Theorem~\ref{theo:classes}) as well as Theorem~\ref{theo:mink}.
Sections~\ref{sec:proofs1} and \ref{sec:proofs2} are devoted to the proof of Theorem~\ref{theo:main}. In Section~\ref{sec:proofs1}, we will treat the `geometric part' and establish the connection between Minkowski contents, S-contents and the growth behaviour of the lengths of the associated fractal string, while in Section~\ref{sec:proofs2}, we finally establish the connection between the geometry of $\Omega$ and its spectral properties. An overview over the various steps of the proof is given in the diagrams on page~\pageref{diagrams}.

\section{Some preliminaries on asymptotic equivalence and similarity} \label{sec:pre}

In this section we collect some basic (and well-known) facts about asymptotic similarity $\asymp$ and asymptotic equivalence $\sim$ of functions, recall the definitions from page \pageref{def:asymp}. Here we formulate all facts for asymptotic similarity/equivalence at $a=0$, but for the case $a=\infty$ they hold analogously.

Consider the family $\sF$ of real valued positive functions $f$ defined on a right neighborhood of $0$, i.e.\ functions $f:(0,y_0)\to(0,\infty)$ for some $y_0>0$. Observe that $\sim$ and $\asymp$ are equivalence relations on $\sF$. In particular, both relations are transitive, i.e.\ if $f_1\asymp f_2$ and $f_2\asymp f_3$ then $f_1\asymp f_3$, as $y\searrow 0$ and similarly for $\sim$. Moreover, since asymptotic equivalence implies asymptotic similarity, we also have  the following implication:
\begin{align} \label{eq:asymp-sim}
  f_1(y)\asymp f_2(y) &\text{ and } f_2(y)\sim f_3(y) \text{ implies }  f_1(y)\asymp f_3(y), \text{ as } y\searrow 0.%\\
  %f_1(y)\sim f_2(y) &\text{ and } f_2(y)\sim f_3(y) \text{ implies }  f_1(y)\sim f_3(y) \text{ as } y\searrow 0
\end{align}
We collect some useful rules for $\asymp$ and $\sim$. Some further rules will be stated later on for functions with extra properties such as regular variation or smoothness (see e.g.\ Lemma~\ref{lemma:8} or \ref{lemma:7}).
\begin{Lemma}
  \label{lem:asymp-properties}
  Let $f_1,f_2,h, g_1, g_2\in\sF$. Then, as $y\searrow 0$,

  \begin{itemize}
    %\item[(i)] $f_1\asymp f_2/h$ if and only if $h\cdot f_1\asymp f_2$;
    \item[(i)] $f_1\sim f_2/h$, if and only if $h\cdot f_1\sim f_2$;
   % \item[(iii)] if $f_1\asymp f_2$ and $g_1\asymp g_2$, then $f_1\cdot g_1 \asymp f_2\cdot g_2$ and $f_1/ g_1 \asymp f_2/ g_2$;
    \item[(ii)] if $f_1\sim f_2$ and $g_1\sim g_2$, then $f_1\cdot g_1 \sim f_2\cdot g_2$ and $f_1/ g_1 \sim f_2/ g_2$;
   % \item[(iv)] Let $(y_j)_{j\in\N}$ be a null sequence of positive numbers. Then $f_1(y)\asymp f_2(y)$, as $y\searrow 0$, implies $f_1(y_j)\asymp f_2(y_j)$ as $j\to\infty$. Moreover, the converse holds, provided $f_1/f_2$ is monotone in a neighborhood of zero.
    \item[(iii)] if $f_1\sim f_2\sim g_1\sim g_2$, then $f_1+f_2\sim  g_1+g_2$.
  \end{itemize}
  The assertions hold analogously with $\sim$ replaced by $\asymp$.
\end{Lemma}

As a first simple application, we provide a proof of Proposition~\ref{lemma:4}. Note that a statement similar to Proposition~\ref{lemma:4} could be formulated for S-contents.

	\begin{proof}[Proof of Proposition~\ref{lemma:4}] The first assertion is a direct consequence of \eqref{eq:asymp-sim} applied to $f_1=V_F$, $f_2=h$ and $f_3=\tilde h$. To see the equality of the upper and lower generalized Minkowski contents, let $\epsilon \in (0,1)$. Then there is a $y_0>0$ such that
	\begin{equation*}
	(1-\epsilon) h(y)<\tilde{h}(y) < (1+\epsilon) h(y)
	\end{equation*}
	for all $y\in (0,y_0]$. Hence we obtain
	\begin{equation*}\label{eq:173}
	0<   \frac{V_F(y)}{(1+\epsilon)h(y)} \leq   \frac{V_F(y)}{\tilde{h}(y)}  \leq \frac{V_F(y)}{(1-\epsilon)h(y)}<\infty
	\end{equation*}
and the assertion follows by taking upper and lower limits in this relation as $y\searrow 0$ and then  letting $\epsilon\searrow 0$.
The last assertion is a direct consequence of the transitivity of $\sim$.
\end{proof}

\section{Regularly varying functions and Karamata theory} \label{section:4}\label{sec:karamata}

In this section we discuss functions of regular variation in more detail. In our exposition we concentrate on those results which turned out to be useful for our purposes and which prepare the proof of our main result. On the way we will also prove Theorems~\ref{theo:classes} and \ref{theo:mink}.
%, a notion dating back to Jovan Karamata in the 1930'ies. Nowadays, the theory of such functions is known as \emph{Karamata theory}.
While the classic theory is formulated for regular variation at $\infty$ (see Definition~\ref{def:reg2} below), we will focus on \emph{regular variation at $0$} (cf.~Definition~\ref{def:reg}). We will reformulate some of the classic results of Karamata theory for this case (which is an easy exercise due to Remark~\ref{rem:1}). For the versions at $\infty$ (which we will also use occasionally) and further details we refer to the monograph \cite{BGT}.

\begin{Definition}\label{def:reg2}
	Let $\varrho\in \R$.
%\begin{enumerate}[(i)]
 % \item
 A function $f :(x_0,\infty) \to (0,\infty)$ (with $x_0\geq 0$) is called \emph{regularly varying (at $\infty$) of index $\varrho$}, if and only if $f$ is measurable and satisfies for all $t>0$
\begin{equation*}
\lim_{x\to\infty}\frac{f(tx)}{f(x)} = t^\varrho.
\end{equation*}
We write $\sR_\varrho[\infty]$ for the family of all such functions.
$f\in\sR_0[\infty]$ are also called \emph{slowly varying (at $\infty$)}.
\end{Definition}

Note that, by Definition~\ref{def:reg}, the domain $\dom(h)$ of any $h\in\sR_\varrho$ is an interval of the form $(0,y_0)$ (a right neighborhood of zero). $h$ can easily be extended to a function on $(0,\infty)$ %(e.g.\ by setting $f(x):=f(x_0)$ for $x\in(0,x_0)$)
without affecting the regular variation property at $0$. It is easy to see that one can even extend $h$ to $(0,\infty)$ in such a way that properties like continuity, differentiability or strict monotonicity (which $h$ may have) are preserved. (Of course, if $\lim_{y\to y_0}h(y)=\infty$, then one has to redefine $h$ near $y_0$ for this extension.) Similarly, any function $f\in\sR_\varrho[\infty]$ (with domain $\dom(f)=(x_0,\infty)$) can be extended to the whole positive axis.

As pointed out above, regularly varying functions (at $\infty$) are well-studied, see e.g.\ \cite{BGT}. The following simple observation allows to transfer virtually all results known for regular variation at $\infty$ to regular variation at $0$.

\begin{rem}\label{rem:1}
 \emph{$h\in\sR_\varrho$ if and only if the function $\tilde h$ defined by $\tilde h(x)=h(1/x)$, $1/x\in\dom(h)$ is in $\sR_{-\varrho}[\infty]$. Indeed, if the domain of $h\in\sR_\varrho$ is $(0,y_0)$ then the domain of $\tilde h$ is the interval $(1/y_0,\infty)$ and for any $t>0$, one obtains
    \begin{equation*}
\lim_{x\to\infty}\frac{\tilde h(tx)}{\tilde h(x)} = \lim_{x\to \infty}\frac{h(1/(tx))}{h(1/x)}=\lim_{y\searrow 0}\frac{h(t^{-1}y)}{h(y)}={t^{-\varrho}},
\end{equation*}
implying $\tilde h\in\sR_{-\varrho}[\infty]$.
}
\end{rem}
 \begin{rem}\label{rem:1-1}
 \emph{
 Observe that condition (H2) in the definition of the classes $\sG_\varrho$ is regular variation at $0$ plus an extra uniformity requirement in the variable $t$. We will see below from Theorem~\ref{theo:UCT} that this local uniformity is automatically satisfied for regularly varying functions.
 In \cite{HeLap2}, property (H2) was called a \emph{homogeneity property}. Indeed, any regularly varying function $h\in \sR_\varrho$ is \emph{asymptotically homogeneous of degree $\varrho$} in the following sense: for any $t>0$,
		\begin{equation*}\label{eq:0}
		h(ty) = h(y) (t^\varrho+o(1)), \text{ as } y\searrow 0.
		\end{equation*}
This relation %\emph{(\ref{eq:0})}
is equivalently described by $h(ty) \sim t^\varrho h(y)$, as $y\searrow 0$.
}
\end{rem}

The following simple observation characterizes regular variation in terms of slow variation, compare also with the Characterization Theorem~\cite[Theorem 1.4.1]{BGT}:

\begin{Lemma} \label{lem:CT}
Let $\varrho\in\R$.
  For any $h\in\sR_\varrho$, there is a slowly varying function $\ell\in\sR_0$ such that $h(y)=y^\varrho \ell(y)$ for any $y\in\dom(h)$.
\end{Lemma}
\begin{proof}
   Define  $\ell$ by $\ell(y):=h(y)/y^\varrho$, $y\in\dom(h)$ and observe that, for any $t>0$, $\ell(ty)/\ell(y)\to 1$, as $y\searrow 0$ implying $\ell\in\sR_0$.
\end{proof}

Our next observation is that the classes $\sR_\varrho$ are stable with respect to asymptotic equivalence. Any function $g$  asymptotically equivalent to some regularly varying function $h$ is already regularly varying itself. In contrast, asymptotic similarity $\asymp$ does not preserve regular variation.  	
	\begin{Lemma}\label{lemma:8}	
	Let $h\in \sR_\varrho$ and let $g:(0,y_0)\to(0,\infty)$ be a measurable function such that
%some asymptotically equivalent measurable function, i.e.
$h(y) \sim g(y),$ as $y\searrow 0$.	Then $g\in \sR_\varrho$, i.e.\ $g$ is regularly varying at $0$ with the same index $\varrho$.
\end{Lemma}
	\begin{proof}
	 The function $z$ defined by $z(y) := \frac{g(y)}{h(y)}$ for all $y\in \dom(h)\cap \dom(g)$ is measurable and converges to 1 as $y\searrow 0$. Therefore,
	\begin{equation}
	\frac{g(ty)}{g(y)}= \frac{h(ty)}{h(y)}%\underbrace{
\frac{z(ty)}{z(y)}%}_{\rightarrow 1, \text{ as } y\searrow 0}
\rightarrow t^\varrho, \text{ as } y\searrow 0,
	\end{equation}
	for all $t>0$. Hence $g$ is regularly varying of index $\varrho$.
	\end{proof}

This simple observation will be used frequently in the proofs later on. It is also the key to Theorem~\ref{theo:mink}, which we prove now.

\begin{proof}[Proof of Theorem~\ref{theo:mink}.] \label{proof:mink}
By the assumptions, there exists some $h\in\sR_\varrho$ and some $M>0$, such that $\sM(h;F)=M$, which means that
$\lim_{r\searrow 0} \frac{V_F(r)}{M\, h(r)}=1$ or, equivalently $V_F(r)\sim M\, h(r)$, as $r\searrow 0$.
Since any (positive) constant multiple of a function in $\sR_\varrho$ is still in $\sR_\varrho$, it follows from Lemma~\ref{lemma:8}, that $V_F\in\sR_\varrho$.
\end{proof}

%\subsection{Karamata Theory}\label{section:3.3}
The Uniform Convergence Theorem (UCT) is one of the central results of Karamata theory.
We formulate a version for functions $h\in\sR_\varrho$. It can easily be derived from the corresponding statement for functions in $\sR_\varrho[\infty]$, cf.\ Remark~\ref{rem:1}.

\begin{Satz}[{UCT, cf.~\cite[Theorem 1.5.2]{BGT} for a version at $\infty$}]\label{theo:UCT}
	Let $\varrho\in\R$ and $h \in \sR_\varrho$. Then $\frac{h(ty)}{h(y)}\rightarrow t^\varrho$, as $y\searrow 0$ converges uniformly in $t$ on any compact subset of $(0,\infty)$.
\end{Satz}
  Theorem~\ref{theo:UCT} implies in particular that condition (H2) is satisfied for any regularly varying function $h\in\sR_\varrho$.
As a further consequence of Theorem~\ref{theo:UCT}, we may now prove the following result on asymptotic similarity/equivalence of regularly varying functions. It will be extremely useful in the proofs later on.
\begin{Lemma}\label{lemma:7}
	Let $\kappa \in \R$ and $g\in \sR_\kappa$. Let $f_1,f_2:(0,\infty)\to(0,\infty)$ such that $\lim_{y\searrow 0} f_2(y)=0$ and  assume $f_1(y) \asymp f_2(y)$, as $y\searrow 0$.
Then %Let further be regularly varying. Then
	\begin{equation} \label{eq:lem7}
	g(f_1(y)) \asymp g(f_2(y)), \text{ as }y\searrow 0.
	\end{equation}
Moreover, if $g\in \sR_0$  or $f_1(y) \sim f_2(y)$, as $y\searrow 0$, then
\begin{equation} \label{eq:lem7-2}
g(f_1(y)) \sim g(f_2(y)), \text{ as }y\searrow 0.
\end{equation}
Completely analogous relations hold, if $\lim_{y\searrow 0} f_2(y)=\infty$ and  $g\in\sR_\kappa[\infty]$, or if $f_1(x)\asymp f_2(x)$, as $x\to\infty$.
\end{Lemma}	
\begin{proof} By assumption, there are positive constants $y_0, c_1, c_2$ such that
$c(y) :=\frac{f_1(y)}{f_2(y)} \in [c_1,c_2]$ for all $y\in (0,y_0)$. Assume first that $\kappa\geq 0$ and fix some $\eps\in(0,c_1^\kappa)$.  By the Uniform Convergence Theorem~\ref{theo:UCT}, the convergence $g(tx)/g(x)\to t^\kappa$ (as $x\searrow 0$) is uniform in $t$ on the compact interval $[c_1,c_2]$, i.e.\ there exists some $\delta>0$ such that $|g(tx)/g(x)-t^\kappa|\le\eps$ for each $0<x<\delta$ and each $t\in[c_1,c_2]$. Therefore, 
\begin{equation*}%\label{eq:lem7-1}
\frac{g(f_1(y))}{g(f_2(y))} = \frac{g(c(y)f_2(y))}{g(f_2(y))} \leq c(y)^\kappa+\eps\leq c_2^\kappa+\eps=:c_2',
\end{equation*}
for each $y\in(0,y_0)$ such that $f_2(y)<\delta$. Observe that the hypothesis on $f_2$ implies there exists some $y_1=y_1(\delta)>0$ such that the $f_2(y)<\delta$ is true for all $0<y<y_1$.
In a similar way one obtains that the expression ${g(f_1(y))}/{g(f_2(y))}$ is bounded from below by $c_1':=c_1^\kappa-\eps$ for all $y\in(0,y_1)$.  This shows \eqref{eq:lem7} for the case $\kappa\geq 0$.

For $\kappa<0$, the only difference is that now the function $t\mapsto t^\kappa$ is decreasing. Fixing some $\eps\in(0,c_2^\kappa)$, a similar argument as in the previous case gives the upper bound $c_2':=c_1^\kappa+\eps$ and the lower bound $c_1':=c_2^\kappa-\eps$ for the expression  $g(f_1(y))/g(f_2(y))$ for all $y\in(0,y_1)$. This completes the proof of the first assertion.

In case $\kappa=0$ the constants used above are $c_1'=1-\eps$ and $c_2'=1+\eps$, where the $\eps>0$ was fixed. Observing that $\eps$ can be chosen as close to zero as we wish, assertion \eqref{eq:lem7-2} follows for $g\in\sR_0$.
Assertion \eqref{eq:lem7-2} is also valid under the hypothesis $f_1\sim f_2$, since in this case there is for any $\eps>0$ some $y_0$ such that $(1-\eps)\leq f_1(y)/f_2(y)\leq (1+\eps)$ for all $y\in(0,y_0]$. This implies
\begin{equation*} %\label{eq:lem7-1}
\frac{g((1-\eps)f_2(y))}{g(f_2(y))} \leq \frac{g(f_1(y))}{g(f_2(y))} \leq \frac{g((1+\eps)f_2(y))}{g(f_2(y))} \to (1+\eps)^\kappa,
\end{equation*}
as $y\searrow 0$. Similarly, the left expression converges to $(1-\eps)^\kappa$, as $y\searrow 0$. Now \eqref{eq:lem7-2} follows by letting $\eps$ tend to $0$.

The arguments for the case $\lim_{y\searrow 0} f_2(y)=\infty$ and $g\in\sR_\kappa[\infty]$, and for $f_1(x)\asymp f_2(x)$, as $x\to\infty$ are completely analogous.
\end{proof}

A second central result of Karamata theory is the following Representation Theorem, which provides an integral representation for slowly varying functions. Again, we formulate it for slow variation at $0$.
% It is equivalent to the Uniform Convergence Theorem \ref{theo:UCT}. %Its proof utilises a local boundedness, and hence local integrability   argument  furnished  by  e.g. E. Seneta (cf. [Se, Lemma 1]).
\begin{Satz}[{Representation Theorem, cf.  \cite[Theorem 1.3.1]{BGT}}]\label{theo:rep}
	A function $\ell$ is slowly varying at $0$, if and only if it has a representation
	\begin{equation}\label{eq:rep}
	\ell(y) = c\left(y\right) \exp\left( \int_{y}^{a} \frac{\varepsilon(u)}{u}du   \right)\text{, } y \in (0,a),
	\end{equation}
	for some $a>0$, where $c$ and $\varepsilon$ are measurable functions with  $\lim_{y\searrow 0} c(y) = C \in (0,\infty)$ and  $\lim_{y\searrow 0}\varepsilon(y) = 0$.
\end{Satz}
\begin{proof}
  Applying \cite[Theorem 1.3.1]{BGT} to the function $\tilde \ell\in\sR_0[\infty]$ given by $\tilde \ell(x):=\ell(1/x)$, yields the existence of some positive constants $\tilde a$ and $C$ and some functions $\tilde c$ and $\tilde \eps$ with $\tilde c(x)\to C$ and $\tilde \eps(x)\to 0$ as $x\to\infty$ such that
  \begin{equation*}%\label{eq:rep}
	\tilde \ell(x) = \tilde c\left(x\right) \exp\left( \int_{\tilde a}^{x} \frac{\tilde \eps(u)}{u}du   \right)\text{, } x\geq\tilde a.
	\end{equation*}
  Therefore, we get for any $0<y<a:=1/\tilde a$,
  \begin{equation*}%\label{eq:rep}
	\ell(y)=\tilde \ell(1/y) = \tilde c\left(1/y\right) \exp\left( \int_{\tilde a}^{1/y} \frac{\tilde \eps(u)}{u}du   \right)=c(y)\exp\left( \int_{1/a}^{1/y} \frac{\tilde \eps(u)}{u}du   \right),
	\end{equation*}
  where $c(y):=\tilde c(1/y)$ satisfies obviously $c(y)\to C$ as $y\searrow 0$. The representation \eqref{eq:rep} follows now with the substitution $u\mapsto 1/u$ and setting $\eps(y):=\tilde \eps(1/y)$.
\end{proof}

Note that the representation of $\ell$ given by Theorem \ref{theo:rep} is not unique, for we may add to $\varepsilon$ any function $\delta$ with $\lim_{y\searrow 0} \delta(y)=0$ (and adjust $c$ accordingly). 
Moreover, the upper bound $a$ of the integration interval in \eqref{eq:rep} may be shifted arbitrarily in $\dom(\ell)$, again by adjusting the function $c$. If $a>1$, for instance, we have
\begin{equation}
\ell(y)= c(y)\exp\left( \int_{y}^a \frac{\varepsilon(u)}{u}du   \right) = \tilde{c}(y)  \exp\left( \int_{y}^1 \frac{\varepsilon(u)}{u}du\right),
\end{equation}
where
\begin{equation*}
\tilde{c}(y) := c(y) \exp\left( \int_{1}^{a} \frac{\varepsilon(u)}{u}du   \right) .
\end{equation*}

Combining Lemma~\ref{lem:CT} with the Representation Theorem~\ref{theo:rep}, we obtain an analogous representation for regularly varying functions:
 Since for $h\in \sR_\varrho$, $\varrho\neq 0$, there is some $\ell\in\sR_0$ such that $h(y)=y^\varrho\ell(y)$, there must exist  functions $c$ and $\eps$ (with $c(y)\to C\in(0,\infty)$ and $\eps(y)\to 0$, as $y\searrow 0$) such that
\begin{equation*} %\label{eq:138}\begin{aligned}
h(y) = y^\varrho \ell(y) = c(y) \exp \left( \varrho\cdot \log y + \int_{y}^{1}  \varepsilon(u) \frac{du}{u}  \right)
= c(y) \exp\left(\int_{y}^{1} -\varrho + \varepsilon(u) \frac{du}{u}  \right)
\end{equation*}
From this representation it is easy to conclude the following asymptotic behaviour.
\begin{Proposition}[{cf. [BGT, Proposition 1.5.1]}]\label{prop:1}
	Let $\varrho\neq 0$ and $h\in \sR_\varrho$. Then
	\begin{align*}
	h(y) \rightarrow \left\{ \begin{array}{rrl}
	0, &\text{if} &\text{$\varrho>0$}\\
	\infty, &\text{ if}&\text{$\varrho<0$}
	\end{array}   \right., \quad \text{ as } y\searrow 0.
	\end{align*}
\end{Proposition}

\begin{rem}\label{rem:8}
\emph{
Let $\Omega\subset\R^d$ be a bounded open set and $F=\bd\Omega$. Assume that $F$ is $h$-Minkowski nondegenerate for some $h\in\sR_{d-D}$, i.e.\ $0<\lsM(h;F)\leq \usM(h;F)<\infty$. Then the Minkowski dimension of $F$ exists and equals $D$.
Indeed, by Lemma~\ref{lem:CT}, we can write $h$ as $h(r)=r^{d-D}\ell(r)$ for some $\ell\in\sR_0$ and so,  for any $\gamma>0$,
  $$
  \sM^{D+\gamma}(F)=\lim_{r \searrow 0}\frac{V_F(r)}{r^{d-D-\gamma}}
  =\lim_{r \searrow 0}\frac{V_F(r)}{r^{d-D}\ell(r)}\,r^\gamma\ell(r)
  \leq\usM(h;F)\,\lim_{r \searrow 0}r^\gamma \ell(r).
  $$
  Since $r\mapsto r^\gamma\ell(r)$ is regularly varying with index $\gamma>0$, by Proposition~\ref{prop:1}, the last limit vanishes, %(while $\usM(h;F)$ is finite)
  implying $\sM^{D+\gamma}(F)=0$ and thus $\overline{\dim}_MF\leq D$.
 Similarly, we get $\sM^{D+\gamma}(F)\geq \lsM(h;F)\,\lim_{r \searrow 0}r^\gamma \ell(r)$, where the last limit is now $+\infty$ for $\gamma<0$, implying $\underline{\dim}_M F\geq D$. This shows $\dim_M F=D$.
Therefore, clearly, the index of any suitable regularly varying gauge function $h$ for $F$ is necessarily equal to its (Minkowski) co-dimension $d-D$. We point out that this does not imply that for every $F$ with dimension $\dim_M F =D$ there exists a suitable gauge function $h\in\sR_{d-D}$.
}
\end{rem}

The following fact will be useful in the sequel, e.g.\ in Theorem~\ref{theo:kara1} below. It is another immediate consequence of the Representation Theorem.
\begin{cor}[{cf.~\cite[Corollary 1.4.2]{BGT}}]\label{cor:9}
\label{lemma:1}
	Let $\varrho\in\R$ and $h\in \sR_\varrho$. Then there is $x_0>0$ such that the functions $h$ and $\frac{1}{h}$ are locally bounded and locally integrable on $(0,x_0]$.
\end{cor}

The Representation Theorem allows to show that any regularly varying function $h \in \sR_\varrho$ of index  $\varrho\in(0,1)$ already satisfies condition (H3) (cf.\ Definition~\ref{def:gd}).
This is another important step towards the asymptotic equivalence of the classes $\sG_D$ and $\sSR_D$.

\begin{Proposition}\label{lemma:H3}
	Let $\varrho \in (0,1)$ and $h \in \sR_\varrho$. Then $h$ satisfies hypothesis (H3), i.e.\ there are constants $m>0$, $\tau\in (0,1)$, $y_0, t_0 \in (0,1]$ such that
	\begin{equation*}
	\frac{h(ty)}{h(y)} \geq m t^\tau,
	\end{equation*}
	for all $0<y \leq y_0$ and $0< t \leq t_0$.
\end{Proposition}

  \begin{proof}
     By Lemma~\ref{lem:CT}, $h(y)=y^\varrho \ell(y)$ for some $\ell\in\sR_0$, and, by the Representation Theorem \ref{theo:rep}, we can write $\ell$ as in equation (\ref{eq:rep}). Therefore, for any $t,y\in(0,1)$,
\begin{align}\label{eq:139}
\frac{\ell(y)}{\ell(ty)} &=
\frac{ c(y)\cdot \exp \left( \int^{1}_{y} \frac{\varepsilon(u)}{u}du    \right)}
{c(ty) \cdot \exp \left( \int^{1}_{ty}\frac{\varepsilon(u)}{u}du   \right)}
=\frac{c(y)}{c(ty)}\exp \left( -\int_{ty}^{y} \varepsilon(u)\frac{du}u\right), %&=\frac{(c+z_1(ty))}{(c+z_1(y))}\exp \left( \int_{\frac{1}{y}}^{\frac{1}{ty}} \frac{\varepsilon(u)-\varrho}{u}du   \right).
\end{align}
where $c,\eps$ are measurable functions such that $c(y)\to C$ (for some $C\in(0,\infty)$) and $\eps(y)\to 0$ as $y\searrow 0$.
Let $\delta>0$ such that $C-\delta>0$ and let $\gamma>0$ such that $\varrho+\gamma<1$. Since, $c(y)\to C$ as $y\searrow 0$, there is $y_0\in(0,1)$ such that $|c(y)-C|< \delta$ for all $y \in (0,y_0)$. Moreover, since $\eps(y)\to 0$, there is also some $y_1\in(0,y_0]$ such that $|\eps(y)|<\gamma$ for any $y\in(0,y_1)$.  Therefore, we obtain for any $y\in(0,y_1)$,
\begin{align}
\frac{\ell(y)}{\ell(ty)} &\leq \frac{C+\delta}{C-\delta}\exp \left( \int_{ty}^{y} |\varepsilon(u)|\frac{du}{u}  \right)
 \leq M \exp \left( \int_{ty}^{y} \gamma\frac{du}u  \right)=M t^{-\gamma},
\end{align}
 where we have set $M := \frac{C+\delta}{C-\delta}$. This implies $\ell(ty)/\ell(y)\geq M^{-1} t^{\gamma}$ and thus
 $$
 \frac{h(ty)}{h(y)}=t^\varrho \frac{\ell(ty)}{\ell(y)}\geq M^{-1} t^{\varrho+\gamma},
 $$
 for all $y\in(0,y_1)$ and all $t\in(0,1)$. Thus the assertion follows for $m:=M^{-1}$ and $\tau:=\varrho+\gamma$, which by the choice of $\gamma$ is strictly less than 1.
\end{proof}

Up to now, we have seen that any regularly varying function $h\in\sR_\varrho$ with $\varrho\in(0,1)$ satisfies conditions (H2) and (H3), by Theorem~\ref{theo:UCT} and Proposition~\ref{lemma:H3}, respectively. Concerning condition (H1), it is easy to see from Proposition~\ref{prop:1} that such $h$ also satisfy $\lim_{y\searrow 0}h(y)=0$ and $\lim_{y\searrow 0}{h(y)}/{y}=\infty$. The condition $\lim_{y\rightarrow \infty}h(y) = \infty$ is not always satisfied but it is not very relevant. One can easily extend or redefine $h$ on some interval $[x_0,\infty)$ bounded away from zero to meet this condition without affecting the asymptotic properties of $h$ at $0$.
%(i.e.\ in a way that the new function $\tilde h$ satisfies $\tilde h(y)\sim h(y)$, as $y\searrow 0$).
On the other hand, regularly varying functions need neither be continuous nor strictly increasing. It will be our aim now to clarify that for each $h\in\sR_\varrho$ there is an asymptotically equivalent function with these two properties. It turns out that there are even smooth representatives for $h$.

\subsection*{Smoothly varying functions.}\label{section:RD}
 We will now discuss the classes $\sSR_\varrho$ %of smoothly varying functions
 in more detail. We will show in particular the asymptotic equivalence of the classes $\sG_\varrho$ (cf.\ Definition \ref{def:gd}) and $\sSR_\varrho$, see Theorem \ref{thm:SR=G}.
 %We call two classes of functions that are connected in such a way \emph{asymptotically equivalent}. Our main result of this section is that the classes $\sG_D$, $\sR_D$ and $\sSR_D$ are asymptotically equivalent for any $D\in(0,1)$.
Recall that for any $\varrho\in\R$, a regularly varying function $h\in \sR_\varrho$ is in the class $\sSR_\varrho$, if and only if there is some $y_0>0$ such that $h \in C^1(0,y_0]$. Similarly, we write $\sSR_\varrho[\infty]$ for the class of functions that are smoothly varying at $\infty$. %in $\sR_\varrho[\infty]\cap C^1[x_0,\infty]$
%We set $\sSR:=\bigcup_{\varrho\in\R}\sSR_\varrho$.

By the following result, we may assume slowly varying functions $\ell\in \sR_0$ to be smooth modulo asymptotic equivalence, implying that the classes $\sSR_0$ and $\sR_0$ are asymptotically equivalent.

\begin{Satz}[{cf.\ \cite[Theorem 1.3.3]{BGT}}]\label{theo:1}
	Let $\ell\in \sR_0$. Then there is a smoothly varying function $\ell_1\in\sSR_0$ asymptotically equivalent to $\ell$  (i.e.\ $\ell_1(y)\sim \ell(y)$, as $y\searrow 0$). $\ell_1$ can be chosen in such a way that all derivatives of the function $p_1$ defined by $p_1(x) := \log(\ell_1(e^{-x})),  e^{-x}\in\dom(\ell_1)$ vanish, i.e., for all $n\in \N$,
 \begin{equation*}
	p_1^{(n)}(x)\to 0, \text{ as } x\rightarrow \infty.
	\end{equation*}
%	where $h_1^{(n)}$ denotes the $n$-th derivative of $h_1$.\\
	
 %Cf. \emph{\cite{BGT}} for the precise construction$\mathrm{)}$
\end{Satz}

Combining Theorem \ref{theo:1}  with Lemma~\ref{lem:CT}, we obtain the following result, which establishes in particular that the classes $\sR_\varrho$ and $\sSR_\varrho$ are asymptotically equivalent for any $\varrho\in\R$.

\begin{cor}\label{prop:3}
		Let $\varrho \in \R$ and $h \in \sR_\varrho$. Then there is a smoothly varying function $h_1 \in\sSR_\varrho$ asymptotically equivalent to $h$ (i.e.\ with $h_1(y)\sim h(y)$, as $y\searrow 0$). $h_1$ can be chosen in such a way that the function $p_1$ defined by $p_1(x) := \log(h_1(e^{-x}))$, $x\in\dom(h_1)$ satisfies
\begin{align*}
		\lim_{x\to\infty} p_1'(x)=-\varrho \quad \text{ and } \quad \lim_{x\to\infty} p_1^{(n)}(x)=0 		\end{align*}
		for each  $n\in \N, n\geq 2$. %where $h_1'$ denotes the derivative and $h^{(n)}$ the $n$-th derivative of $h_1$.
\end{cor}
\begin{proof} By Lemma~\ref{lem:CT}, %the Characterisation Theorem \ref{theo:char},
the function  $h$ can be written as $h(y) = y^\varrho \ell(y)$ for some $\ell\in\sR_0$. Theorem \ref{theo:1} implies the existence of a smoothly varying function $\ell_{1}\in\sSR_0$ asymptotically equivalent to $\ell$. Defining $h_1$ by $h_1(y) := y^\varrho \ell_{1}(y), y>0$, we obtain
\begin{equation*}
p_1(x) =\log(h_1(e^{-x})) = \log(e^{-\varrho x}\ell_{1}(e^{-x})) = -\varrho x +\log (\ell_{1}(e^{-x})).
\end{equation*}
By Theorem \ref{theo:1}, all derivatives of the function $x\mapsto\log (\ell_{1}(e^{-x})), x>0$ tend to zero, as $x\rightarrow \infty$, which yields the assertion.
\end{proof}

The next statement characterizes the asymptotic behavior of the the derivative of a smoothly varying function. It is not only needed in the proof of the subsequent proposition but it will also be extremely useful in the proofs in Section~\ref{sec:proofs1}.
The expression $E_h(y):=\frac{y\, h'(y)}{h(y)}$ is known as the \emph{elasticity} of a function $h$ at $y$.

\begin{Lemma} \label{lem:elasticity}
For any $\varrho\in\R$ and any function $h\in\sSR_\varrho$, the elasticity of $h$ at $y$ converges to $\varrho$ as $y\searrow 0$, i.e.
  \begin{equation}\label{eq:M-S-eq}
	\lim_{y\searrow 0} \frac{y h'(y)}{h(y)} =\varrho.
	\end{equation}
For $\varrho\neq 0$, this can be rephrased as $h'(y)\sim \varrho\, h(y)/y$, as $y\searrow 0$.
\end{Lemma}
\begin{proof}
It is enough to show the assertion for slowly varying functions. Indeed, for $h\in\sSR_\varrho$, there is, by Lemma~\ref{lem:CT}, some function $\ell\in\sSR_0$ such that $h(y)=y^\varrho\ell(y)$ on $\dom(h)$ and therefore
$$
E_h(y)=\frac{y\,h'(y)}{h(y)}=\frac{y\,(y^\varrho\ell(y))'}{y^\varrho\ell(y)}=\frac{\varrho y^\varrho\ell(y) +y^{\varrho+1}\ell'(y) }{y^\varrho\ell(y)}=\varrho+\frac{y\,\ell'(y)}{\ell(y)}=\varrho+ E_\ell(y).
$$
That is, the elasticities of $h$ and $\ell$ differ pointwise by $\varrho$ and so in particular, as $y\searrow 0$, $E_h(y)\to\varrho$ if and only if $E_\ell(y)\to 0$.

So let $h\in\sSR_0$.  For any fixed $y\in\dom(h)$, we apply the linear substitution $u(t)=yt-y$ in the differential quotient of $h$ at $y$ and obtain
\begin{align*}
\lim_{y\searrow 0} \lim_{u\searrow 0} \frac{h(y+u)-h(y)}{u}\frac{y}{h(y)} &=\lim_{y\searrow 0} \lim_{t\rightarrow 1}\frac{h(ty)-h(y)}{y(t-1)h(y)} y\\ &= \lim_{t\rightarrow 1}\frac{1}{t-1}\lim_{y\searrow 0} \frac{h(ty)-h(y)}{h(y)}=0,%\right),
\end{align*}
where in the last expression we have changed the order of the limits. This is justified, since, by the Uniform Convergence Theorem \ref{theo:UCT}, the inner limit in this last expression vanishes uniformly in $t$ on any compact interval. (So choose one containing $1$ in its interior.) It follows that also the outer limit (as $t\to 1$) of this expression exists (and equals 0), implying that it is safe to interchange the order of the limits.
\end{proof}
Our next observation is that smoothly varying functions are monotone near zero. One of the consequences is that the derivatives of smoothly varying functions are again regularly varying.
%That is es that the derivative of a smoothly varying function has again some kind of regular variation property (though it does not fall into the class of positive functions, in general).  It
\begin{Proposition}\label{lemma:essgauge}  \label{prop:smooth-monotone}
   Let $\varrho\neq 0$ and $h \in \sSR_\varrho$. %\marginpar{Gilt nun fuer alle $\varrho\neq0$! CHECK!}
  If $\varrho>0$, then $h$ is strictly increasing in some right neighborhood of $0$, and if $\varrho<0$, then $h$ is strictly decreasing in some right neighborhood of $0$.
  In particular, the function $\varrho^{-1} h'$ (when restricted to this neighborhood) is in the class $\sR_{\varrho-1}$.
\end{Proposition}
\begin{proof}
   Let $\varrho\neq 0$ and $h\in\sSR_\varrho$. Fix some $\epsilon\in(0,1)$ and let $(0,y_0)$ be an interval in which $h$ is differentiable. Then \eqref{eq:M-S-eq} in Lemma~\ref{lem:elasticity} implies that there is some $y_1\in(0,y_0]$ such that for each $y\in(0,y_1)$,
   $$
   (1-\epsilon)\varrho h(y)/y<h'(y)<(1+\epsilon)\varrho h(y)/y.
   $$
   For $\varrho>0$, the expression on the left is strictly positive, which means in particular that $h'$ is strictly positive in the interval $(0,y_1)$. Hence the function $h$ is strictly increasing in $(0,y_0)$. For $\varrho<0$, the expression on the right is strictly negative, which implies that $h'$ is strictly negative in  $(0,y_1)$. Hence $h$ is strictly decreasing in $(0,y_1)$.
   For the last assertion note that $\varrho^{-1} h'$ is positive in the interval $(0,y_1)$. Moreover, by Lemma~\ref{lem:elasticity}, we have $\varrho^{-1} h'\sim h(y)/y$, as $y\searrow 0$. Therefore, $\varrho^{-1} h'\in\sR_{\varrho-1}$ follows from Lemma~\ref{lemma:7}, since the function $y\mapsto h(y)/y$ is in $\sR_{\varrho- 1}$.
   This completes the proof.
\end{proof}
 We have gathered all the results necessary in order to show the asymptotic equivalence of the classes $\sSR_\varrho$ and $\sG_\varrho$.%, (cf.~{Definition~\ref{def:gd}}).
	\begin{Satz} \label{thm:SR=G}
Let $\varrho\in (0,1)$.
For any function $h\in \sG_\varrho$, there is a smoothly varying function $\tilde h \in \sSR_\varrho$ asymptotically equivalent to $h$, %$($i.e.\ such that $f(y)\sim h(y)$, as $y\searrow 0)$
and vice versa. Hence, the classes $\sG_\varrho$ and $\sSR_\varrho$ are asymptotically equivalent.	
\end{Satz}
\begin{proof} %\marginpar{Ueberarbeiten!}
    %The class $\sSR_\varrho$ is not the empty set, since it contains at least $X^\varrho$. \\
%\begin{itemize}
	%\item
 Let $\varrho\in (0,1)$ and $h\in \sG_\varrho$. By condition (H2), $h$ is regularly varying (at $0$), and therefore, by Corollary~\ref{prop:3}, there exists some function $\tilde h \in \sSR_\varrho$ such that $\tilde h(y)\sim h(y)$, as $y\searrow 0$.

To show the converse, let $\tilde h\in \sSR_\varrho$.  Then $\tilde h$ satisfies hypothesis (H2) by the Uniform Convergence Theorem \ref{theo:UCT}. Furthermore, $\tilde h$ satisfies (H3) by Proposition~\ref{lemma:H3}. It remains to show that the function $\tilde h$  satisfies hypothesis (H1) modulo asymptotic equivalence.

The function $\tilde h$ is by definition $C^1$-smooth and therefore in particular continuous on some interval $(0,y_0]$ for some $y_0>0$. By Proposition~\ref{prop:smooth-monotone}, $\tilde h$ is also strictly increasing on the interval $(0,y_1]$, for some $y_1\in(0,a]$. Therefore, $\tilde h$ can now easily be extended or redefined on the interval $(y_1,\infty)$ in such a way that it becomes continuous and strictly increasing on $(0,\infty)$ and satisfies $\lim_{y\to\infty} \tilde h(y)=\infty$ (e.g.\ by setting $h(y) := (y-y_1) + \tilde h(y_1)$ for $y>y_1$ and $h(y):=\tilde h(y)$ on $(0,y_1]$). Note that this will not affect the asymptotic properties at $0$, i.e.\ for the new function $h$, we have  $h(y)\sim\tilde h(y)$ as $y\searrow 0$ and $h\in\sR_\varrho$. In particular, $h$ still satisfies (H1) and (H2).
By the Lemma~\ref{lem:CT}, we have the representation $h(y) = y^\varrho \ell(y)$ for some  $\ell \in \sR_0$ %Note that $\ell$ is smooth (in $(0,y_0)$), because $h$ is.
and so it is easy to see from Proposition~\ref{prop:1} that
\begin{equation*}%\begin{array}{llll}
\lim_{y\searrow 0} \frac{ h(y)}{y} =
\lim_{y\searrow 0} y^{\varrho-1} \ell(y) = \infty \quad \text{ and }
 \quad \lim_{y\searrow 0} h(y) = 0,
\end{equation*}
since $\varrho\in(0,1)$. Therefore, we have found a function $h$, asymptotically equivalent to $\tilde h$, satisfying all the conditions of the class $\sG_\varrho$. This completes the proof.
\end{proof}

Combining Corollary~\ref{prop:3} and Theorem~\ref{thm:SR=G}, we conclude that for any $\varrho\in(0,1)$ the three classes $\sR_\varrho$, $\sSR_\varrho$ and $\sG_\varrho$ are asymptotically equivalent, hence Theorem~\ref{theo:classes} is proved. Regular variation is indeed the essential property which makes things work.
 %It is convenient and at the same time astonishing that modulo asymptotic equivalence smoothly varying functions of index $D\in(0,1)$ already satisfy all the hypotheses of the class $\sG_D$.
 Therefore, it is only natural to expect that in Theorem~\ref{theo:1.2} the class $\sG_{1-D}$  can be substituted with the class $\sR_{1-D}$. However, in order to do this, one needs to clarify that the auxiliary functions $f$ and $g$ associated in Theorem~\ref{theo:1.2} to any $h\in\sG_{1-D}$ have a well defined counterpart for any $h\in\sR_{1-D}$. This could be achieved by using the concept of generalized asymptotic inverses (see e.g.~\cite[\S 1.5.7]{BGT}), %\marginpar{Add chapter/section from [1]!},
 which, however, would impose additional technical difficulties. The latter can be avoided by restricting to smoothly varying functions, which, due to the asymptotic equivalence of the classes $\sR_{1-D}$ and $\sSR_{1-D}$, is rather a convenience than a restriction.
  In fact, we will not only see that the functions $f$ and $g$ are well-defined for any $h\in\sSR_{1-D}$, but also that they inherit from $h$ regular variation and smoothness.
 % it remains to see that we may invert regularly varying functions with \textit{regularly varying inverse functions}.

 For this purpose we let $D\in(0,1)$ and $h\in\sSR_{1-D}$ (with $1-D\in(0,1)$). We define the function $H$ by
 $$
 H(y):= y/h(y), y\in\dom(h)
 $$
 and observe that  $H \in \sSR_{D}$. By Proposition~\ref{prop:smooth-monotone}, $H$ is strictly increasing on some interval $(0,y_0)$ and therefore it can be inverted on this interval. The inverse function $H^{-1}$ is then well-defined on the interval $(0,H(y_0))$ (and may be extended beyond this interval in some arbitrary way, if needed). $H^{-1}$ inherits the properties of smoothness and strict monotonicity from $H$. Moreover, observing that for any $z\in(0,H(y_0))$ there is a unique $y\in(0,y_0)$ such that $z=H(y)$ (and that $z$ depends continuously on $y$), we obtain for any $t>0$,
 $$
 \lim_{z\searrow 0} \frac{H^{-1}(tz)}{H^{-1}(z)}
 =\lim_{y\searrow 0}\frac{H^{-1}(t H(y))}{H^{-1}(H(y))}=\lim_{y\searrow 0}\frac{H^{-1}(H(t^{1/D}y))}{H^{-1}(H(y))}=t^{1/D},
 $$
 where we have used for the second equality that $H\in\sR_{D}$. Hence, we have shown that $H^{-1}$ is smoothly varying with index $1/D$, i.e.\ $H^{-1}\in\sSR_{1/D}$. We are now ready to define the functions $f$ and $g$.

\begin{Proposition} \label{prop:g-and-f}
  Let $D\in(0,1)$ and $h\in\sSR_{1-D}$. Then there is some $x_0\geq 0$ such that the functions
  $g$ and $f$ are well-defined  for any $x\in[x_0,\infty)$ by
$$
  g(x):=H^{-1}(1/x) \quad \text{ and } \quad f(x):=\frac 1{H(1/x)}=x h(1/x),
  $$
  and $C^1$-smooth on this interval. Moreover, $g\in\sSR_{-1/D}[\infty]$ and  $f\in\sSR_{D}[\infty]$.
\end{Proposition}
\begin{proof}
  By the considerations above, $H^{-1}$ is well-defined on $(0,H(y_0)]$, where $y_0$ is chosen such that $h$ (and thus $H$) is strictly increasing and $C^1$-smooth on $(0,y_0]$. Letting $x_g:= 1/H(y_0)$, we have $1/x\in (0, H(y_0)]$ for any $x\in [x_g,\infty)$, and so $H^{-1}(1/x)$ and thus $g(x)$ are well-defined and smooth on $[x_g,\infty)$.
  Since $H^{-1}\in\sR_{1/D}$, it follows directly from Remark~\ref{rem:1} that $g\in\sR_{-1/D}[\infty]$.

  Similarly, $f$ is well-defined and $C^1$ on $[x_f,\infty)$ for $x_f:=1/y_0$ whenever $H$ is $C^1$ on $(0,y_0]$. Moreover, $y\mapsto 1/H(y)$ is regularly varying (at $0$) with index $-D$, since $H$ has index $D$. Thus, again by Remark~\ref{rem:1}, we obtain $f\in\sSR_D[\infty]$.
  (The actual assertion of the statement is satisfied for $x_0:=\max\{x_g,x_f\}$, however, we will not need a common interval for $g$ and $f$ in our applications.)
\end{proof}

 In the proof of Theorem \ref{theo:1.2} given in \cite{HeLap2}, hypothesis (H3) is a technical assumption in order to leisurely apply Lebesgue dominated convergence. Karamata theory, however, allows to circumvent this kind of reasoning due to, among others, the following powerful result, known as Karamata's Theorem. As we will only need a version for functions regularly varying at $\infty$, we directly restate the corresponding result in \cite{BGT} for this class of functions.

\begin{Satz}[Karamata's Theorem; direct half {\cite[Theorem 1.5.11]{BGT}}]\label{theo:kara1}

	Let $\varrho\in \R$ and $f \in \sR_\varrho[\infty]$. Let further  $X>0$ such that $f$ is locally bounded in $[X,\infty)$, cf.\ Cor.~\ref{lemma:1}. Then
	\begin{enumerate}
		\item [$(i)$] For any $\sigma \geq -(\varrho+1)$
		\begin{equation*}
		x^{\sigma+1}f(x)\left/\int_{X}^{x}u^\sigma f(u)du \rightarrow \sigma+\varrho+1 \text{, as } x\rightarrow \infty. \right.
		\end{equation*}
		\item [$(ii)$]For any $\sigma < -(\varrho+1)$ $($and for $\sigma = -(\varrho+1)$ if $\int_\cdot^\infty u^{-(\varrho+1)}f(u)du <\infty)$
		\begin{equation*}
		x^{\sigma+1} f(x)\left/\int_{x}^{\infty}u^\sigma f(u)du\rightarrow -(\sigma+\varrho+1) \text{, as } x\rightarrow \infty. \right.
		\end{equation*}
	\end{enumerate}
\end{Satz}

One of the most interesting consequences of Karamata's Theorem \ref{theo:kara1} in the context of regularly varying gauge functions is that we may take slowly varying functions $\ell\in \sR_0[\infty]$ out of  integrals in the following fashion: %\marginpar{Wo gebraucht?}
\begin{equation}\label{eq:out-of-int}
\int_{x}^{\infty} u^\varrho \ell(u)du\sim \ell(x)\int_{x}^{\infty}u^\varrho du, \text{ as } x\to\infty,
\end{equation}
whenever $\varrho<-1$ (and this is the case we will need). Indeed, this follows easily from part (ii) of Karamata's Theorem. Moreover, this theorem provides the following useful relations. We point out that statements similar to (i) and (ii) below have been proved and used by He and Lapidus for functions related to the classes $\sG_\varrho$, see \cite[Proposition 3.2 and Lemma 3.3]{HeLap2}.

\begin{Proposition}\label{prop:quotg}
	Let $\varrho<-1$ and $g\in \sR_{\varrho}[\infty]$. Then%-\frac{1}{1-D} be regularly varying.
	\begin{enumerate}
\item[(i)] % \begin{equation*}
	 $\displaystyle \int_{x}^{\infty}g(u)du\sim  -\frac{1}{\varrho+1}xg(x)$, $\text{ as } x\to \infty$;
	%\end{equation*}
		\item [(ii)] for any $t>0$,	
\begin{equation*}
	\frac{\int_{tx}^{\infty}g(u)du}{\int_{x}^{\infty}g(u)du} \to t^{\varrho+1}, \text{ as } x\rightarrow \infty,
	\end{equation*}
	where the convergence is uniform in $t$ on any compact subset of $(0,\infty)$;
	\item [(iii)]
$\displaystyle
\sum_{j=k}^\infty g(j)\sim -\frac{1}{\varrho+1}kg(k), \text{ as } k \to \infty.
$
	\end{enumerate}
\end{Proposition}
%\newpage$

\begin{proof}
Assertion (i) follows directly from Karamata's Theorem \ref{theo:kara1} (ii) for $\sigma =0$.

For a proof of (ii) observe that the function $x\mapsto x g(x), x\in\dom(g)$ is regularly varying (at $\infty$) with  index $\varrho+1$. Hence, by Lemma~\ref{lemma:8} and  assertion (i), the function $G(x):=\int_{x}^{\infty}g(u)du$, $x\in\dom(g)$ is in $\sR_{\varrho+1}[\infty]$. Therefore, the convergence $G(tx)/G(x)\to t^{\varrho+1}$, as $x\to \infty$ is obvious and the uniformity follows from the Uniform Convergence Theorem~\ref{theo:UCT} (version at $\infty$).

It remains to prove (iii). By Lemma~\ref{lem:CT}, there is some $\ell\in\sR_0[\infty]$ such that $g(x)=x^\varrho\, \ell(x)$, $x\in\dom(g)$. Define $\tilde g$ by $\tilde g(x):=g(j)$, for $x\in[j,j+1)\cap\dom(g)$ and $j\in\N$,
%and $\tilde g(x):=x^\varrho \tilde\ell(x)$.
and $\tilde \ell$ by $\tilde \ell(x):=\tilde g(x)/x^\varrho$, $x\in\dom(g)$. We claim that
%$\tilde g\in\sR_\varrho[\infty]$ or equivalently
$\tilde g\in\sR_\varrho[\infty]$ or, equivalently, $\tilde \ell\in\sR_0[\infty]$.
To prove this, by Lemma~\ref{lemma:8}, it suffices to show that $\tilde\ell(x)\sim\ell(x)$, as $x\to\infty$.
First observe that, for any $x\in[j,j+1)\cap\dom(g)$,
$
\tilde \ell(x)=g(j)/x^\varrho
$
and thus, since $x\mapsto x^\varrho$ is decreasing,
$$
\ell(j)=\frac{g(j)}{j^\varrho}\leq \tilde\ell(x)\leq \frac{g(j)}{(j+1)^\varrho}=\left(\frac j{j+1}\right)^\varrho \ell(j).
$$
Since $j=[x]$ and thus $j/(j+1)\to 1$ as $x\to\infty$, we conclude $\tilde\ell(x)\sim\ell([x])$, as $x\to\infty$. Finally, we note that $\ell([x])\sim \ell(x)$, as $x\to\infty$, since $\ell\in\sR_0[\infty]$. Indeed, letting $t_x:=[x]/x$ and noting that $1/2\leq 1-1/x\leq t_x\leq 1$ for any $x\geq 2$, the Uniform Convergence Theorem\ref{theo:UCT} (applied to $\ell$ for $t$ on the compact interval $[1/2,1]$) yields
$$
\lim_{x\to\infty}\frac{\ell([x])}{\ell(x)}=\lim_{x\to\infty}\frac{\ell(t_x\,x)}{\ell(x)}=1.
$$
This completes the proof of our claim that $\tilde g\in\sR_\varrho[\infty]$.
Now assertion (iii) follows directly by applying (i) to $\tilde g$, since, for any $k\in\N\cap\dom(g)$,
\begin{align*}
\sum_{j=k}^\infty g(j)&=\sum_{j=k}^\infty \tilde g(j)=\int_k^\infty \tilde g (u) du.\qedhere
%=\int_k^\infty u^\varrho \tilde \ell(u) du.
\end{align*}
 \end{proof}

\section{The geometric part of the proof} \label{sec:proofs1}

In this section we discuss the equivalence of the first three assertions in Theorem \ref{theo:main} as well as that of assertions (vi),(vii) and (viii). The direct proofs of (i) $\Leftrightarrow$ (iii) and (vi) $\Leftrightarrow$ (viii) given in \cite{HeLap2} are rather technical, cf.~\cite[Theorems 3.4,  3.8, 3.10 and 4.1]{HeLap2}. Using characterization results for Minkowski contents in terms of S-contents from \cite{RW09,RW13}, does not only allow to add another equivalent criterion to each of the two parts of Theorem~\ref{theo:main}, but also to simplify the proofs significantly. Instead of proving (i) $\Leftrightarrow$ (iii) directly, we will establish (i) $\Leftrightarrow$ (ii) and (ii) $\Leftrightarrow$ (iii) separately. Similarly we will show (vi) $\Leftrightarrow$ (vii) and (vii) $\Leftrightarrow$ (viii).  The (generalized) S-contents, which describe the behavior of the boundary measure of the parallel sets, provide thus an extremely useful connecting link between Minkowski contents (volume of the parallel sets) and the growth of the lengths in the associated fractal string.

\subsection{S-contents vs. Minkowski contents.} Our first aim is to verify the equivalences (i) $\Leftrightarrow$ (ii) and (vi) $\Leftrightarrow$ (vii) in Theorem~\ref{theo:main}. It turns out that they can be derived essentially from the results in \cite{RW13}.  Therefore, we will not reprove the equivalence here, but rather explain the minor modifications necessary in the relevant results of \cite{RW13} to cover our present situation. In fact, we can establish this equivalence for any bounded open set $\Omega\subset\R^d$ and any gauge function $h\in\sSR_{d-D},  D\in(0,d)$. There is no need to restrict to subsets of $\R$ for this result.
%It is not even necessary that $F=\bd\Omega$ is the boundary of some domain $\Omega$.

\begin{thm}\label{thm:M-S-equiv}
Let $\Omega\subset\R^d$ be bounded and $F=\bd \Omega$. Let $h\in\sSR_{d-D}$ for some $D\in[0,d)$. Then the following assertions are equivalent:
\begin{enumerate}
  \item [(i)] $0< \lsM(h;F) \leq  \usM(h;F) < \infty$,
  \item [(ii)]$0< \lsS(h';F) \leq  \usS(h';F) < \infty$.
\end{enumerate}
In particular, these assertions imply $\dim_M F=D$.
\end{thm}

\begin{rem} \label{rem:full-contents}
   \emph{
   Recall that the (generalized) Minkowski contents and S-contents appearing in the statement above as well as in Theorem~\ref{theo:11} below are defined relative to the set $\Omega$, cf.\ \eqref{eq:h-Mink-def} and \eqref{eq:h-S-def}. The results from \cite{RW13} that we are going to use in the proofs are formulated for the `full' contents (with the set $\Omega$ in the definitions \eqref{eq:h-Mink-def} and \eqref{eq:h-S-def} omitted). However, due the fact that $\Omega$ is metrically associated with its boundary $F$ and that therefore the volume function $r\mapsto V_F(r)=|F_r\cap\Omega|_d$ is a Kneser function, all the results in \cite{RW13} hold literally for the relative contents used here. For more details we refer to the discussion of relative contents in  \cite{W16}.
   }
\end{rem}

\begin{proof}
  In view of Remark~\ref{rem:full-contents}, the stated equivalence follows essentially by combining \cite[Theorem 3.2]{RW13} (where the easier implication (ii) $\Rightarrow$ (i) is established) with \cite[Theorem 3.4]{RW13} (where the reverse implication is obtained). While in \cite[Theorem 3.2]{RW13} $h$ is assumed to be differentiable with derivative $h'$ being non-zero in some right-neighborhood of $0$, an assumption met for any $h\in\sSR_{d-D}$ due to Proposition~\ref{lemma:essgauge}, %(as long as $0<d-D<1$, i.e.\ $d-1<D<d$), \marginpar{$0<d-D<1$ wegen Prop.\ 4.15 - Geht das allgemeiner?}
  there are additional assumptions in \cite[Theorem 3.4]{RW13}: $h$ is assumed to be of the form $h(y)=y^{d-D} g(y)$ with $g$ being non-decreasing and
\begin{equation*}%\label{eq:M-S-eq}
	\limsup_{y\searrow 0} \frac{y h'(y)}{h(y)} <\infty.
	\end{equation*}
The latter assumption is satisfied due to Lemma~\ref{lem:elasticity}, which says that the above limit exists and equals $d-D$ %$\lim_{y\searrow 0} y h'(y)/h(y)= d-D$
for $h\in\sSR_{d-D}$.  Moreover, by Lemma~\ref{lem:CT}, $h$ is clearly of the form $h(y)=y^{d-D}g(y)$ for some $g\in\sSR_{0}$. But the slowly varying factor $g$ is not necessarily non-decreasing. However, inspecting the proof of \cite[Theorem 3.4]{RW13} it is rather easy to see that `non-decreasing' can be replaced by `slowly varying' in this statement. The monotonicity of $g$ is only used once in the proof of \cite[Proposition 3.3]{RW13} to ensure that $g(ar)$ is bounded from below by $g(r)$ for some $a>1$ and all sufficiently small $r>0$.
If $g\in\sR_0$ and $\epsilon>0$, then we can certainly find some $r_1>0$, such that $g(ar)\geq (1-\epsilon)g(r)$ for any $r\in(0,r_1)$ and this suffices to extend the argument in the proof of \cite[Proposition 3.3]{RW13} to functions $g\in\sR_0$. Since \cite[Theorem 3.4]{RW13} is essentially a direct application of \cite[Proposition 3.3]{RW13}, also this statement extends to functions $g\in\sSR_0$.
For the assertion $\dim_M F = D$ see Remark~\ref{rem:8}.
This completes the proof.
\end{proof}

Also the following statement is a direct consequence of a result in \cite{RW13} and the elasticity properties derived in Lemma~\ref{lem:elasticity}.
It establishes the equivalence (vi) $\Leftrightarrow$ (vii) in Theorem~\ref{theo:main}, i.e.\ the equality of the $h$-Minkowski content and the $h'$-S-content.

\begin{Satz}[{cf.~\cite[Theorem 3.7]{RW13}\label{theo:11}}]
	Let $\Omega\subset\R^d$ be bounded and $F=\bd \Omega$. Let $h\in\sSR_{d-D}$ for some $D\in[0,d)$. Suppose $M>0$. Then
	\begin{equation*}
	{\cal{M}}(h;F) = M, \quad \text{ if and only if }\quad {\cal{S}}(h';F)=M.
	\end{equation*}
	Moreover, in this case $\dim_M F = D$.	
\end{Satz}
\begin{proof}
  By Lemma~\ref{lem:CT}, $h$ can be written in the form $h(y)=g(y)y^{d-D}$ for some $g\in\sSR_0$ and by Lemma~\ref{lem:elasticity}, we have $\lim_{y\to 0} g'(y)y/g(y)=0$. Therefore, taking into account Remark~\ref{rem:full-contents}, the hypothesis of \cite[Theorem 3.7]{RW13} is satisfied except that the function $g$ is now slowly varying and not necessarily non-decreasing.  However, we can argue similarly as in the proof of Theorem~\ref{thm:M-S-equiv} above that the forward implication in \cite[Theorem 3.7]{RW13} follows essentially from \cite[Proposition 3.6]{RW13}, which is a refinement of \cite[Proposition 3.3]{RW13} and in which the estimates can easily be modified to work for slowly varying $g$.
  The reverse implication in \cite[Theorem 3.7]{RW13} (which is the reverse implication in Theorem~\ref{theo:11}) is a direct consequence of \cite[Theorem 3.2]{RW13}, which can be applied since, by Proposition~\ref{lemma:essgauge}, $h'$ is non-zero in some right neighborhood of $0$.
  Finally, for a proof of $\dim_M F = D$ see Remark~\ref{rem:8}.
\end{proof}

\subsection{S-Contents and fractal strings.}
Recall that in dimension $d=1$, we can associate to any bounded open set $\Omega\subset\R$, its fractal string ${\sL}=(l_j)_{j\in\N}$, encoding the lengths $l_j$ of the connected components $I_j$ of $\Omega$. They appear in $\sL$ in non-increasing order, i.e.\ $l_1\geq l_2\geq l_3\geq \ldots$, and according to their multiplicities.
We will discuss now the relations between S-contents and the asymptotic growth of the lengths $l_j$, and establish in particular the equivalences $(ii)\Leftrightarrow(iii)$ and $(vii)\Leftrightarrow(viii)$ of Theorem~\ref{theo:main}. Note that (and with $F=\bd\Omega$), the  `\emph{surface area}' of $\bd F_\eps\cap \Omega$ in the definition of the S-content reduces in dimension $d=1$ to the counting measure $\Ha^0$.
Therefore, the key idea behind these two relations is the following simple geometric observation (everything else is `\emph{asymptotic calculus}'): All the intervals $I_j$ of $F$ with length $l_j\leq2\eps$ are completely covered by $F_\eps$ and do not contribute to the boundary $\bd F_\eps$, while each of the remaining ones contributes exactly two points.
 %(except for the two unbounded ones which contribute one point each).
 That is, for any $j\in\N$ and $\eps\in[l_j/2,l_{j-1}/2)$, we have
\begin{align}
  \label{eq:bd-measure} \Ha^0(\bd F_\eps\cap\Omega)=2(j-1).
\end{align}
Recall from Proposition~\ref{prop:g-and-f} that for $D\in(0,1)$ and $h\in\sSR_{1-D}$, the function $g$ is given by $g(x) := H^{-1}(1/x)$, where $H^{-1}$ is the inverse of $H(y) := {y}/{h(y)}$.  %$f(x) = \frac{1}{H(x^{-1})}$

\begin{Satz}\label{theo:4}
	Let $\Omega\subset \R$ be open and bounded, $F:=\bd \Omega$ and  ${\sL}=(l_j)_{j\in\N}$ the associated fractal string. Let $D\in (0,1)$ and $h\in \sSR_{1-D}$. Then the following assertions are equivalent:
\begin{enumerate}
  \item[(i)]$0< \underline{{\cal{S}}}(h';F)
	\leq \ol{{\cal{S}}}(h';F) <\infty,$
  \item[(ii)] $l_j \asymp g(j)$, as $j \to \infty$.
\end{enumerate}
\end{Satz}
\begin{proof}
  We first reformulate both assertions using `asymptotic calculus'. Then their equivalence will follow easily from \eqref{eq:bd-measure}.

  On the one hand, assertion (i) can be rewritten as $\Ha^0(\bd F_\eps\cap \Omega)\asymp h'(\eps)$, as $\eps\searrow 0$, which, by Lemma~\ref{lem:elasticity}, is equivalent to $\Ha^0(\bd F_\eps\cap \Omega)\asymp h(\eps)/\eps$, as $\eps\searrow 0$.

On the other hand, since $g(x)=H^{-1}(1/x)$ and $H\in\sSR_{D}$, we can apply $H$ to both sides of (ii) to see that, by Lemma~\ref{lemma:7}, (ii) is equivalent to
   $$H(l_j)=l_j/h(l_j)\asymp 1/j, \text{ as } j\to\infty.$$ By Lemma~\ref{lem:asymp-properties}, this can also be written as $j\asymp h(l_j)/l_j$, as $j\to\infty.$ %Multiplying by 2 and
   Using the asymptotic homogeneity of $h\in\sSR_{1-D}$ (which implies $h(y)\sim 2^{1-D}h(y/2)$, as $y\searrow 0$, cf.~Remark~\ref{rem:1-1}) and recalling that positive constants do not matter in `$\asymp$'-relations, this is equivalently given by $2j\asymp {h(l_j/2)}/(l_j/2)$, as  $j\to\infty$. Since $(j-1)/j\to 1$, as $j\to\infty$, we can replace $j$ by $j-1$ on the left. Therefore, the statement of the theorem is equivalent to
\begin{align} \label{eq:theo4-reformulated}
  \Ha^0(\bd F_\eps\cap\Omega)\asymp \frac{h(\eps)}{\eps}, \text{ as } \eps\searrow 0, \quad\text{ iff }\quad 2(j-1)\asymp \frac{h(r_j)}{r_j}, \text{ as } j\to\infty,
\end{align}
where $r_j:=l_j/2$, $j\in\N$. ($r_j$ is the `inradius' of an interval of length $l_j$.)

The forward implication in \eqref{eq:theo4-reformulated} is obvious: since \eqref{eq:bd-measure} implies in particular that $\Ha^0(\bd F_{r_j}\cap\Omega)=2(j-1)$ for each $j\in\N$, we just need to plug in the sequence $(r_j)_{j\in\N}$ for $\eps$ in the left assertion.

For a proof of the reverse implication in \eqref{eq:theo4-reformulated}, assume the right hand side holds, which means, there are constants $c_1, c_2$ such that $c_1\leq 2(j-1)/(h(r_j)/r_j)\leq c_2$ for each $j$. Since the function $\eps\mapsto h(\eps)/\eps$ is in $\sSR_{-D}$, it is strictly decreasing in some right neighborhood $(0,\eps_0)$ of $0$, cf.\ Proposition~\ref{lemma:essgauge}. Moreover, $\lim_{\eps\searrow 0} h(\eps)/\eps=+\infty$. Therefore, for any sufficiently large $j\in\N$ (such that $r_{j-1}<\eps_0$) and any $\eps\in [r_j,r_{j-1})$, %(i.e.\ $j$ sufficiently large),
the quotient $\frac{\Ha^0(\bd F_\eps\cap\Omega)}{h(\eps)/\eps}$ is bounded from above and below by
$$
c_1\leq\frac{2(j-1)}{h(r_{j})/r_{j}}\leq \frac{\Ha^0(\bd F_\eps\cap\Omega)}{h(\eps)/\eps}\leq \frac{2(j-1)}{h(r_{j-1})/r_{j-1}}\leq c_2+\frac{2}{h(r_{j-1})/r_{j-1}},%=\frac{2(j-1)}{h(l_{j-1}/2)/(l_{j-1}/2)}+\frac{2}{h(l_{j-1}/2)/l_{j-1}}
$$
where the last summand on the right vanishes as $j\to\infty$. But this implies the left assertion in \eqref{eq:theo4-reformulated}, completing the proof of Theorem~\ref{theo:4}.
\end{proof}

A very similar argument allows to establish the equivalence of generalized S-measurability of $F=\bd\Omega$ and generalized `$L$-measurability' of the associated fractal string, i.e. the equivalence of the assertions (vii) and (viii) in Theorem~\ref{theo:main}.

\begin{Satz}\label{theo:13}
Let $F\subset \R$ be a compact set and let ${\sL} =(l_j)_{j\in \N}$ be the associated fractal string. Let $D\in(0,1)$ and $h\in \sSR_{1-D}$. Let $L>0$. Then the following assertions are equivalent:
\begin{enumerate}[(i)]
\item $F$ is $h'$-S measurable (i.e.\ $0< {\cal{S}}(h',F)	<\infty$) with 	${\cal{S}}(h',F)=S:=\frac{2^{1-D} L^D}{1-D}$,
\item	$l_j \sim L\cdot g(j)$,  as $j\rightarrow \infty$.
\end{enumerate}
\end{Satz}
\begin{proof}
   We follow the line of the argument in the proof of Theorem~\ref{theo:4} and start by reformulating the statement.
   In view of \eqref{eq:theo4-reformulated} and using again $r_j:=l_j/2$, our first claim is that the relation (i)$\Leftrightarrow$(ii) is equivalent to the following equivalence:
   \begin{align} \label{eq:theo4-reformulated3}
  \Ha^0(\bd F_\eps\cap\Omega)\sim 2^{1-D} L^D \frac{h(\eps)}{\eps}, \text{ as } \eps\searrow 0,  \text{ iff }\quad  2(j-1)\sim 2^{1-D}L^D \frac{h(r_j)}{r_j}, \text{ as } j\to\infty.
\end{align}
   Indeed, assertion (i) in Theorem~\ref{theo:13} means $\Ha^0(\bd F_\eps\cap\Omega)\sim S\cdot h'(\eps)$, as $\eps\searrow 0$ and, by Lemma~\ref{lem:elasticity}, $h'(\eps)$ can be replaced by $(1-D)h(\eps)/\eps$ (since $h\in\sSR_{1-D}$), showing the equivalence of (i) and the left assertion in \eqref{eq:theo4-reformulated3}. (Recall that $(1-D)S=2^{1-D} L^D$.) Similarly, by applying $H$ to both sides of assertion (ii) and recalling that $g(x)=H^{-1}(1/x)$, we can infer from Lemma~\ref{lemma:7} that (ii) is equivalent to $H(L^{-1} l_j)=L^{-1}l_j/h(L^{-1}l_j)\sim 1/j$, as $j\to\infty$. Taking into account the homogeneity of $h$ (cf.\ Remark~\ref{rem:1-1}) and Lemma~\ref{lem:asymp-properties}, this can be rephrased as
   $j\sim L^D h(l_j)/l_j$, as $j\to\infty$, which is easily seen to be equivalent to the right assertion in \eqref{eq:theo4-reformulated3}, using again the homogeneity of $h$. This completes the proof of the above claim. It is therefore sufficient to prove the equivalence in \eqref{eq:theo4-reformulated3}.

   The forward implication in \eqref{eq:theo4-reformulated3} is again obvious from \eqref{eq:bd-measure}, by plugging in the sequence $(r_j)_{j\in\N}$ for $\eps$ in the left assertion.
The reverse implication in \eqref{eq:theo4-reformulated3} requires now a slightly refined argument. Assume the right hand side holds. Then, for each $\delta>0$, there is some $j_0=j_0(\delta)$ such that
$$
1-\delta\leq \frac{2(j-1)}{c\cdot h(r_j)/r_j}\leq 1+\delta$$ for each $j\geq j_0$, where $c:=2^{1-D}L^D$. Since the function $\eps\mapsto h(\eps)/\eps$ is in $\sSR_{-D}$, it is strictly decreasing in some right neighborhood $(0,\eps_0)$ of $0$, cf.\ Proposition~\ref{lemma:essgauge}. % and, moreover, $\lim_{\eps\searrow 0} h(\eps)/\eps=+\infty$.
Therefore, for any sufficiently large $j\in\N$ (such that $j\geq j_0$ and $r_{j-1}/2<\eps_0$) and any $\eps\in [r_j,r_{j-1})$, %(i.e.\ $j$ sufficiently large),
the quotient $\frac{\Ha^0(\bd F_\eps\cap\Omega)}{c\cdot h(\eps)/\eps}$ is bounded from above and below by
\begin{align*}
  1-\delta\leq\frac{2(j-1)}{c\cdot h(r_{j})/r_{j}}\leq \frac{\Ha^0(\bd F_\eps\cap\Omega)}{c\cdot h(\eps)/\eps}&\leq \frac{2(j-1)}{c\cdot h(r_{j-1})/r_{j-1}}
  \leq 1+\delta +\frac{2}{c \cdot h(r_{j-1})/r_{j-1}},%=\frac{2(j-1)}{h(l_{j-1}/2)/(l_{j-1}/2)}+\frac{2}{h(l_{j-1}/2)/l_{j-1}}
\end{align*}
where again the last summand on the right vanishes as $j\to\infty$. Since the argument works for any $\delta>0$, the left assertion in \eqref{eq:theo4-reformulated3} follows, completing the proof of Theorem~\ref{theo:13}.
\end{proof}

\section{The `spectral' part of the proof of Theorem \ref{theo:main}}
\label{sec:proofs2}
In this section, we will finally establish the connection between the geometric and the spectral properties of $\Omega$, i.e. in particular the equivalence of the assertions (iii) and (iv) in part I of  Theorem~\ref{theo:main} and the validity of \eqref{eq:main-N} in part II.
Because of the results in the previous section, we can use a combination of the assertions (i)-(iii) of Theorem~\ref{theo:main} to conclude the validity of (iv) -- we will use (ii) and (iii).  In contrast, it is enough to show that (iv) implies at least one of the assertions (i)-(iii).

It will be convenient now to use the \textit{string counting function} $J$ defined  by
\begin{equation}\label{def:J}
J(\varepsilon) := \max \{ j \text{ }|\text{ } l_j > \varepsilon  \},
\end{equation}
for any fractal string $\sL=(l_j)_{j\in\N}$, counting the number of lengths $l_j$ in $\sL$ that are strictly larger than $\varepsilon$.
Since
\begin{align} \label{eq:H0-J-relation}
2J(2\eps)=\Ha^0(\bd F_\eps\cap\Omega),
\end{align}
 for any $\eps>0$, it is easy to see that the asymptotics of $\Ha^0(\bd F_\eps\cap\Omega)$ (described by the S-content) determines the asymptotics of $J(\eps)$, and vice versa.

\begin{Proposition}
  \label{prop:J-asymp}
  Under the hypothesis of Theorem~\ref{theo:main}, any of the assertions (i),(ii) and (iii) in Theorem~\ref{theo:main} %(i.e.\ $0< \lsS(h';F) \leq  \usS(h';F) < \infty$)
  is equivalent to
\begin{align} \label{eq:J-asymp}
  J(\eps)\asymp h(\eps)/\eps, \quad \text{ as } \eps\searrow 0.
\end{align}
Similarly, any of the assertions (vi), (vii) or (viii) in Theorem~\ref{theo:main} %(i.e.\ the S-content $\sS(h';F)$ exists and equals some $S\in(0,\infty)$)
is equivalent to
\begin{align} \label{eq:J-sim}
  J(\eps)\sim L^D h(\eps)/\eps, \quad \text{ as } \eps\searrow 0,
\end{align}
where $L$ is the constant in (viii). In particular, this implies $J\in\sR_{-D}$. %are related (as in Theorem~\ref{theo:13}) by $S=\frac{2^{1-D} L^D}{1-D}.$
\end{Proposition}
\begin{proof}
  Due to the results of the previous section, it suffices to relate \eqref{eq:J-asymp} to assertion (ii) and \eqref{eq:J-sim} to (vii), which is easy to do with the help of \eqref{eq:H0-J-relation}. For the first equivalence recall from the proof of Theorem~\ref{theo:4} (see \eqref{eq:theo4-reformulated}) that assertion (ii) in Theorem~\ref{theo:main} is equivalent to $\Ha^0(\bd F_\eps\cap\Omega)\asymp h(\eps)/\eps, \text{ as } \eps\searrow 0$ and apply \eqref{eq:H0-J-relation}.

  For the second stated equivalence, recall from the proof of Theorem~\ref{theo:13} (see \eqref{eq:theo4-reformulated3}) that assertion (vii) is equivalent to $\Ha^0(\bd F_\eps\cap\Omega)\sim 2^{1-D} L^D h(\eps)/\eps$ as $\eps\searrow0$. %, where $L$ is given by $S=\frac{2^{1-D} L^D}{1-D}$.
  Therefore, by \eqref{eq:H0-J-relation}, $2J(2\eps)\sim 2^{1-D} L^D h(\eps)/\eps$, as $\eps\searrow0$. Substituting $2\eps$ by $\eps$ and taking into account the asymptotic homogeneity of $h$ (of degree $1-D$), we obtain
  $$
  J(\eps)\sim 2^{-D} L^D \frac{h(\eps/2)}{\eps/2}\sim 2^{1-D} L^D \frac{h(\eps) (1/2)^{1-D}}\eps=L^D\frac{h(\eps)}\eps, \text{ as } \eps\searrow 0,
  $$
  which completes the proof of \eqref{eq:J-sim}. Finally, since $\eps\mapsto h(\eps)/\eps$ is regularly varying with index $-D$,  $J\in\sR_{-D}$ follows from \eqref{eq:J-sim} by Lemma~\ref{lemma:8}.
\end{proof}
The following observation will turn out to be very useful in the proof of the implication (iii)$\Rightarrow$(iv).

\begin{Lemma}
  \label{lem:sumA}
  Assume that assertion (iii) of Theorem~\ref{theo:main} holds. Then
  \begin{align*} %\label{eq:sumA}
    \sum_{j> J(2\eps)} l_j\asymp h(\eps), \text{ as } \eps\searrow 0.
  \end{align*}
\end{Lemma}
\begin{proof}
   Due to the assumption $l_j\asymp g(j)$ as $j\to \infty$, we find positive constants $j_0,\underline{\alpha},\overline{\alpha}$ such that $l_j/g(j)\in[\underline{\alpha},\overline{\alpha}]$ for all $j\geq j_0$. Fix $\eps_0>0$ small enough such that $J(2\eps_0)\geq j_0$. Then  we have for any $0<\eps\leq\eps_0$
  $$
  \underline{\alpha}\sum_{j> J(2\eps)} g(j)\leq \sum_{j> J(2\eps)} l_j\leq \overline{\alpha}\sum_{j> J(2\eps)} g(j).
  $$
  This shows
\begin{align}
  \label{eq:sumA}
  \sum_{j> J(2\eps)} l_j\asymp \sum_{j> J(2\eps)} g(j), \quad \text{ as } \eps\searrow 0.
\end{align}
  Since $J(2\eps)\to\infty$ as $\eps\searrow 0$ and $g\in\sSR_{-1/D}$ (with $-1/D<-1$), we infer from Proposition~\ref{prop:quotg} (iii) that
  $$
  \sum_{j> J(2\eps)} g(j)\sim \frac D{1-D} J(2\eps) g(J(2\eps)), \quad \text{ as } \eps\searrow 0.
  $$
 % Hence
%   $$
%  \sum_{j\geq J(2\eps)} l_j\asymp J(2\eps) g(J(2\eps)), \quad \text{ as } \eps\searrow 0.
%  $$
  Using \eqref{eq:J-asymp}, it follows from Lemma~\ref{lemma:7} and the definition of $g$ that the right hand side is asymptotically similar to
  $  \frac{h(\eps)}{\eps} g(\frac{h(\eps)}{\eps})=h(\eps)$, as $\eps\searrow 0$.
  Combining this with \eqref{eq:sumA}, the assertion of the lemma follows.
\end{proof}

Now we are ready to reformulate and prove the implication (iii) $\Rightarrow$ (iv) in Theorem \ref{theo:main}. It is convenient to employ the function
\begin{align} \label{eq:tilde-delta1}
  \tilde \delta(2\eps):=\delta(\frac 1{2\eps})=\sum_{j=1}^\infty \left\{\frac {l_j}{2\eps}\right\}, \quad \eps>0.
\end{align}

\begin{Satz}%[{cf. [LaHe, Theorem 3.5]}]
\label{theo:8}
	Let $D\in (0,1)$ and $h \in \sSR_{1-D}$. Let $(l_j)_{j\in\N}$ be a fracal string with  $l_j \asymp g(j)$, as $j\rightarrow \infty$. Then
	\begin{equation} \label{eq:tilde-delta2}
\tilde \delta(2\eps)\asymp \frac{h(\eps)}{\eps}, \quad \text{ as } \eps\searrow 0,
	%\frac{\alpha}{\beta^D}\frac{1-D}{D}\leq \underline{\delta}\leq \ol{\delta} \leq \frac{\beta}{\alpha^D} \frac{1-D}{D}+\beta^{1-D},
\end{equation}
or, equivalently,
	%\begin{equation}
$\delta(x)\asymp f(x), \text{ as } x\to\infty$,
%	\frac{\alpha}{\beta^D}\frac{1-D}{D}\leq \underline{\delta}\leq \ol{\delta} \leq \frac{\beta}{\alpha^D} \frac{1-D}{D}+\beta^{1-D},
%\end{equation}
where $f$ and $g$ are given as in Proposition~\ref{prop:g-and-f}.
	\end{Satz}
 \begin{proof}
The equivalence of the two assertions is obvious from the substitution $x=1/2\eps$ and the definition of $f$. Therefore, it suffices to prove \eqref{eq:tilde-delta2}.
We split  $\tilde \delta$ as follows
\begin{equation} \label{eq:tilde-delta3}
\tilde \delta(2\eps) = \sum_{j=1}^\infty \{\frac {l_j}{2\eps}\} = \sum_{j>J(2\eps)}\{\frac {l_j}{2\eps}\} + \sum_{j\leq J(2\eps)}\{\frac {l_j}{2\eps}\}.
\end{equation}
Now observe that for $j>J(2\eps)$, we have $l_j/2\eps<1$ and thus $\{l_j/2\eps\}=l_j/2\eps$. Therefore, we can employ Lemma~\ref{lem:sumA} to the first sum on the right and infer that
$$
\sum_{j>J(2\eps)}\{\frac {l_j}{2\eps}\}=\frac {1}{2\eps}\sum_{j>J(2\eps)}l_j\asymp \frac{h(\eps)}{2\eps}, \quad \text{ as } \eps\searrow 0.
$$
Since $0\leq \{y\}<1$, we infer for the second sum on the right of \eqref{eq:tilde-delta3} that, for any $\eps>0$,
$$
0\leq  \sum_{j\leq J(2\eps)}\{\frac {l_j}{2\eps}\}\leq J(2\eps),
$$
and by \eqref{eq:J-asymp} in Proposition~\ref{prop:J-asymp}, $J(2\eps)$ is bounded above by $c\cdot  h(\eps)/ \eps$ for some $c>0$. %(Recall that, by Theorem~\ref{theo:4},  the assumption implies (iii) of Theorem~\ref{theo:main} and thus \eqref{eq:J-asymp}.)
Combining the estimates of both sums, we conclude that $\tilde \delta(2\eps)$ is bounded from above and below by some constant multiple of $h(\eps)/\eps$ for all sufficiently small $\eps>0$, proving assertion \eqref{eq:tilde-delta2}.
\end{proof}

Our next aim is to show that $h$-Minkowski measurability (i.e.\ any of the assertions (vi), (vii), (viii) in Theorem~\ref{theo:main}) implies the exact asymptotic second term of the eigenvalue counting function as stated in \eqref{eq:main-N}. For the main argument we follow closely the idea of the proof in \cite{LapPo1} for the case of Minkowski measurable sets (which is also used in \cite{HeLap2}) to split the sum $\delta(x)$ in a very special way into three summands and then estimate each of them separately. In the estimation part Karamata theory turns out to be very useful again, allowing a simpler argument as in \cite{HeLap2}. %However, in this part our argument only underlines the ingenious idea of Lapidus' and Pomerance's original proof in \cite{LapPo1}.
Our first step is a refinement and generalization of Lemma~\ref{lem:sumA} above.

\begin{Lemma}
  \label{lem:sumA-fine}
  Assume that assertion (viii) of Theorem~\ref{theo:main} holds. Then, for any $k\in\N$,
  \begin{align*} %\label{eq:sumA}
    \sum_{j> J(k\eps)} l_j\sim \frac{D L^{D}}{1-D} k^{1-D}\, h(\eps), \text{ as } \eps\searrow 0.
  \end{align*}
\end{Lemma}
\begin{proof}
   Due to the assumption $l_j\sim L\, g(j)$ as $j\to \infty$, we find for any $\gamma >0$ some $j_0>0$ such that $l_j/(L\,g(j))\in[1-\gamma,1+\gamma]$ for all $j\geq j_0$. Given $k\in\N$, fix $\eps_0=\eps_0(k)>0$ small enough such that $J(k\eps_0)\geq j_0$. Then we have for any $0<\eps\leq\eps_0$
  $$
  (1-\gamma)L\sum_{j> J(k\eps)} g(j)\leq \sum_{j> J(k\eps)} l_j\leq (1+\gamma)L\sum_{j> J(2\eps)} g(j).
  $$
  Since we can find such $j_0$ and $\eps_0$ for any $\gamma>0$, we infer
\begin{align*}
  %\label{eq:sumA}
  \sum_{j> J(k\eps)} l_j\sim L \sum_{j> J(k\eps)} g(j), \quad \text{ as } \eps\searrow 0,
  \end{align*}
  for any fixed $k\in\N$.
  Since $J(k\eps)\to\infty$ as $\eps\searrow 0$ and $g\in\sSR_{-1/D}$ (with $-1/D<-1$), we infer from Proposition~\ref{prop:quotg} (iii) that, for any $k\in\N$,
  $$
  \sum_{j> J(k\eps)} g(j)\sim \frac{D}{1-D} J(k\eps) g(J(k\eps)), \quad \text{ as } \eps\searrow 0.
  $$
 % Hence
%   $$
%  \sum_{j\geq J(2\eps)} l_j\asymp J(2\eps) g(J(2\eps)), \quad \text{ as } \eps\searrow 0.
%  $$
  Applying \eqref{eq:J-asymp} and taking into account the definition of $g$ (cf.~Proposition~\ref{prop:g-and-f}) and Lemma~\ref{lemma:7}, we infer
  $$
  J(k\eps) g(J(k\eps))\sim L^D \frac{h(k\eps)}{k\eps} g\left(L^D \frac{h(k\eps)}{k\eps}\right)
  \sim L^{D-1}\frac{h(k\eps)}{k\eps} g\left(\frac{h(k\eps)}{k\eps}\right)=L^{D-1}\,h(k\eps),
  $$
  as  $\eps\searrow 0$. Combining the last three estimates and using the homogeneity of $h$,
  the assertion of the lemma follows.
\end{proof}

Now we are ready for the main step in the proof of an exact asymptotic second term of $N$. We will show that (viii) in Theorem~\ref{theo:main} implies an exact estimate for the packing defect $\delta$.

\begin{Satz}%[{cf. [LaHe, Theorem 3.5]}]
\label{theo:J-delta}
	Let $D\in (0,1)$ and $h \in \sSR_{1-D}$. Let $(l_j)_{j\in\N}$ be a fractal string with  $l_j \sim L g(j)$, as $j\rightarrow \infty$. Then
	\begin{equation} \label{eq:delta2}
\delta(x)\sim -\zeta(D)L^D f(x), \quad \text{ as } x\to 0,
	%\frac{\alpha}{\beta^D}\frac{1-D}{D}\leq \underline{\delta}\leq \ol{\delta} \leq \frac{\beta}{\alpha^D} \frac{1-D}{D}+\beta^{1-D},
\end{equation}
%	\frac{\alpha}{\beta^D}\frac{1-D}{D}\leq \underline{\delta}\leq \ol{\delta} \leq \frac{\beta}{\alpha^D} \frac{1-D}{D}+\beta^{1-D},
%\end{equation}
where $g$ and $f$ are given as in Proposition~\ref{prop:g-and-f}.
	%where $\ol{\delta}$ and $\underline{\delta}$ are defined as in equation \emph{(\ref{eq:d2})}.\\
%	\noindent Recall that  $H = \frac{X}{h}$,  $f(x) = \frac{1}{H(x^{-1})}$ and $g(x) = H^*(x^{-1})$, where $H^*\in \sSR_{\frac{1}{1-D}}$ is a smoothly varying asymptotic inverse of $H\in \sSR_{1-D}$ $($cf. \emph{Theorem \ref{theo:2}}$)$.
	\end{Satz}
 \begin{proof}
For fixed $k\in\N$, we split  $\delta$ as follows
\begin{equation} \label{eq:delta3}
 \delta(x) = \sum_{j>J(1/x)}\{l_jx\} + \sum_{q=2}^{k} \sum_{j= J(q/x)+1}^{J((q-1)/x)} \{l_jx\}+\sum_{j\leq J(k/x)} \{l_j x\}
\end{equation}
Now observe that $J(q/x)<j\leq J((q-1)/x)$ implies $[l_jx]=q-1$, and therefore the second sum equals
 \begin{align*}
     \sum_{q=2}^{k} \sum_{j= J(q/x)+1}^{J((q-1)/x)} (l_jx- (q-1))
 &=\sum_{j=J(k/x)+1}^{J(1/x)} l_jx -\sum_{q=2}^{k}(q-1)\left(J\left(\frac{q-1}x\right)-J\left(\frac qx\right)\right)\\
 &=x \sum_{j=J(k/x)+1}^{J(1/x)} l_j -\sum_{q=1}^{k-1}J\left(\frac{q}x\right)+(k-1)J\left(\frac kx\right).
 \end{align*}
Combining the first sum of this last expression with the first sum in \eqref{eq:delta3}, in which all terms satisfy $l_jx<1$ implying  $\{l_jx\}=l_jx$, we obtain
\begin{align*}
  \delta(x) = x \sum_{j>J(k/x)}l_j + \left(kJ\left(\frac kx\right)-\sum_{q=1}^{k-1}J\left(\frac{q}x\right)\right) +\sum_{j\leq J(k/x)} \{l_j x\}- J\left(\frac kx\right),
\end{align*}
for each $k\in\N$ and each $x>0$. We will see in a moment, that the first two terms
$$
A:= x \sum_{j>J(k/x)}l_j
\quad \text{ and }\quad
B:= kJ\left(\frac kx\right)-\sum_{q=1}^{k-1}J\left(\frac{q}x\right)
$$
contribute asymptotically to $\delta(x)$ as $x\to\infty$ for any $k\in\N$, while the remainder term
$$
C:=\sum_{j\leq J(k/x)} \{l_j x\}- J\left(\frac kx\right)=\sum_{j\leq J(k/x)} \left(\{l_j x\}-1\right)
$$
will vanish as $k\to\infty$. (Note that $A, B, C$ depend on $x$ and $k$.) Up to here we followed the argument of He and Lapidus in the proof of \cite[Theorem~4.3]{HeLap2} (which is similar to the one in the proof of \cite[Theorem 4.2]{LapPo1}), the estimates are now derived in an easier way using Karamata theory.

First, by Lemma~\ref{lem:sumA-fine}, we get for $A$:
$$
A\sim \frac{D L^{D}}{1-D} k^{1-D}\, x h(1/x)= L^{D}\,\frac{D}{1-D} k^{1-D}\, f(x), \quad\text{ as } x\to\infty. $$
Second, since $J\in\sR_{-D}$ by Proposition~\ref{prop:J-asymp}, the homogeneity implies in particular that $J(q/x)\sim q^{-D} J(1/x)$, as $x\to\infty$, for any $q>0$ and together with \eqref{eq:J-sim} we infer that
$$
B\sim \left(k^{1-D} -\sum_{q=1}^{k-1}q^{-D}\right)J\left(\frac{1}x\right)
\sim L^D \left(k^{1-D} -\sum_{q=1}^{k-1}q^{-D}\right)f(x), \quad \text{ as } x\to\infty.
$$
Combining the estimates for $A$ and $B$ we conclude that, for each $k\in\N$,
\begin{align}
   \label{eq:A+B}
   \frac{A+B}{L^D f(x)}\to  \frac{1 }{1-D} k^{1-D} -\sum_{q=1}^{k-1}q^{-D}= w_k(D)+\frac 1{1-D} \text{ as } x\to\infty,
\end{align}
where the function $w_k$ is defined for each $s\in\C$ by
$$
w_k(s):=\int_1^k(t^{-s}-[t]^{-s}) dt \quad \left(=-\frac 1{1-s} +\frac 1{1-s}k^{1-s}-\sum_{q=1}^{k-1} q^{-s}, s\neq 1\right).
$$
The functions $w_k$ are entire and, as $k\to\infty$, they converge (uniformly on any compact subset of $\Re(s)>0$) to the function
$$
w(s):=\int_1^\infty(t^{-s}-[t]^{-s}) dt.
$$
$w$ is analytic in $\Re(s)>0$ and known to satisfy the relation $w(s)=-1/(1-s)-\zeta(s)$ for  $\Re(s)>0$, see e.g.~\cite[eq.~(2.3)]{LapPo1}. In particular, this implies
\begin{align}
  \label{eq:zeta}
  w_k(D)+\frac 1{1-D}\to -\zeta(D), \text{ as } k\to\infty.
\end{align}
It remains to estimate $C$. The obvious relation $-1\leq{l_j x}-1<0$ implies $-J(k/x)\leq C\leq 0$ for any $k\in\N$. Using again Proposition~\ref{prop:J-asymp}, we infer that $J(k/x)\sim k^{-D} J(1/x)\sim L^D k^{-D} f(x)$, as $x\to\infty$. Hence the expression $-L^D k^{-D} f(x)$ is essentially a lower bound for $C$. More precisely, %if we fix any $\epsilon>0$, there is some $x_0=x_0(k)$ such that, for any $x>x_0$,
\begin{align}
  \label{eq:C} - k^{-D}= \lim_{x\to\infty}\frac{-J(k/x)}{L^D f(x)}\leq \liminf_{x\to\infty}\frac{C}{L^D f(x)}\leq \limsup_{x\to\infty}\frac{C}{L^D f(x)}\leq 0.
\end{align}
Combining now the estimates for $C$ in \eqref{eq:C} and $A+B$ in \eqref{eq:A+B}, and recalling that $\delta(x)=A+B+C$, we infer that, for each $k\in\N$,
$$
 w_k(D)+\frac 1{1-D}- k^{-D}\leq \liminf_{x\to\infty}\frac{\delta(x)}{L^D f(x)}\leq \limsup_{x\to\infty}\frac{\delta(x)}{L^D f(x)}= w_k(D)+\frac 1{1-D}
$$
Letting now $k\to\infty$, the left and right expression both converge to $-\zeta(D)$, which completes the proof of the theorem.
\end{proof}

It remains to show that assertion (iv) in Theorem \ref{theo:main} implies at least one of the assertions (i), (ii), (iii) (and therefore all of them by the results in Section~\ref{sec:proofs1}). Due to the equivalence of \eqref{eq:J-asymp} with (ii), the following statement is essentially the implication (iv) $\Rightarrow$ (ii).
\begin{Satz} \label{theo:delta-J}
Let $D \in (0,1)$ and $h\in \sSR_{1-D}$. Let $\sL=(l_j)_{j\in\N}$ be a fractal string such that $\delta(x) \asymp f(x)$, as $x\rightarrow \infty$.
 %of finite length $\sum_{j\in \N} l_j <\infty$. Further let $\delta$ be  defined as in \emph{(\ref{eq:d1})} and $h\in \sSR_D$ be smoothly varying, such that $\delta(x) \asymp f(x)$, as $x\rightarrow \infty$.
 Then $J(1/x) \asymp f(x)$, as $x\rightarrow \infty$.
%\noindent Recall that  $H = \frac{X}{h}$,  $f(x) = \frac{1}{H(x^{-1})}$ and $g(x) = H^*(x^{-1})$, where $H^*\in \sSR_{\frac{1}{1-D}}$ is a smoothly varying asymptotic inverse of $H\in \sSR_{1-D}$ $($cf. \emph{Theorem \ref{theo:2}}$)$.
\end{Satz}
\begin{proof}
By assumption, there are positive constants $a_1, a_2$ and $x_0$ such that
\begin{equation}\label{eq:152}
a_1 f(x) \leq \delta(x) \leq a_2 f(x),
\end{equation}
for all $x\geq x_0$.

\noindent \emph{Part 1}: We first derive a lower bound for $J$. Fix some small $\epsilon_0>0$ and choose some integer $k \geq 2$  with
\begin{equation*}
k > \left( \frac{2 a_2(1+\epsilon_0)}{a_1}  \right)^{\frac{1}{1-D}}, \text{ i.e.\ such that } \quad k^{D}<\frac{a_1}{2a_2(1+\epsilon_0)}k.
\end{equation*}
Since $f\in \sSR_{D}[\infty]$, there is an $x_1\geq x_0>0$ such that
\begin{equation}\label{eq:115}
k^{D}(1-\epsilon_0) \leq \frac{f(kx)}{f(x)} \leq k^{D} (1+\epsilon_0) < \frac{a_1}{2a_2}k,
\end{equation}
for all $x\geq x_1$, where the last inequality is due to the choice of $k$. %Without loss of generality, we may assume that $x_0\geq x_2(\geq x_1)$.

We split the packing defect $\delta(x)$ into two sums as follows
\begin{equation}\label{eq:116}
\delta(x) = \sum_{\{l_jx  \} < k^{-1} } \{l_j x\}+ \sum_{\{l_jx\} \geq k^{-1}} \{l_jx\}=:U(x)+V(x).
\end{equation}
Observe that for any $\gamma>0$ such that $k \{ \gamma \}<1$, we have $k\{ \gamma  \} = \{ k\gamma \}$. Indeed, since $k [\gamma]\in\N$, we get
\begin{equation*}
1 > k\{ \gamma  \}= \{ k \{ \gamma  \} \} = \{ k ( \gamma - [\gamma]  ) \} = \{ k\gamma -  k [\gamma]\} = \{ k \gamma \}.
\end{equation*}
Applying this to the sum $U$, we obtain for $x\geq x_1$,
\begin{equation*}
kU(x)  {=} \sum_{\{l_jx\} < k^{-1} } \{ l_jkx  \}  < \delta(kx)\overset{(\ref{eq:152})}{\leq} a_2 f(kx) \overset{(\ref{eq:115})}{\leq} \frac{a_1}{2}kf(x) \overset{(\ref{eq:152})}{\leq}\frac{1}{2}k\delta(x).
\end{equation*}
Thus $U(x)\leq \frac{1}{2}\delta(x)$, which together with $U(x)+V(x) = \delta(x)$ implies that  $$V(x)\geq \frac{1}{2}\delta(x)$$ for all $x\geq x_1$.
On the other hand, since every summand in $V$ is less than 1, we have
$$
V(x)\leq \sum_{\{l_jx\} \geq k^{-1}} 1\leq \#\{j:{\{l_jx\} \geq k^{-1}}\}\leq \#\{j:l_jx \geq k^{-1}\}=J\left(\frac 1{kx}\right).
$$
Combining both estimates, we obtain for $x\geq x_1$
$$
J\left(\frac 1{kx}\right)\geq \frac{1}{2}\delta(x)\geq \frac{1}{2}a_1f(x),
$$
which provides a lower bound on $J(1/x)$ in terms of $f(x)$ as desired. Substituting $kx$ by $x$ and using again the homogeneity \eqref{eq:115}, we conclude that
\begin{equation}\label{eq:117}
J\left(\frac{1}{x}\right){\geq}\frac{a_1}{2}f(x/k)\geq \frac{a_1}{2}k^{-D}(1-\epsilon_0)f(x),
\end{equation}
for all $x \geq kx_1$.

\emph{Part 2}: In order to derive an upper bound for $J$, we first prove the following estimate: there are positive constants $c, x_2$ such that, for each $x\geq x_2$,
\begin{equation}\label{eq:121}
J\left(\frac{1}{x}\right)-J\left(\frac{2}{x}\right) \leq c\, f(x).
\end{equation}
%Let $X>0$ and $x\in [X,\infty)$.
For a proof of \eqref{eq:121} consider the interval $ I:=(J(\frac{2}{x}),J(\frac{1}{x})] $. Let $J_\sigma\subset I$ denote the subset of integers $j$ in $I$  with $\{l_jx\} \geq \frac{1}{2}$, and $J_\kappa\subset I$ denote the set of those integers $j$ in $I$ with $\{l_jx\} < \frac{1}{2}$. We write $\sigma := \#J_\sigma$ and $\kappa := \#J_\kappa$ for the corresponding numbers of elements. Then $\sigma + \kappa = J(\frac{1}{x})-J(\frac{2}{x})$. On the one hand, we see that
\begin{equation}\label{eq:118}
\delta(x) = \sum_{j\in \N} \{l_jx\} \geq \sum_{j \in J_{\sigma}} \{l_jx\}\geq \sum_{j\in J_{\sigma}} \frac{1}{2}= \frac{1}{2}\sigma.
\end{equation}
On the other hand,  $J(\frac{2}{x}) < j \leq J(\frac{1}{x})$ implies $ \frac{2}{x} > l_j \geq \frac{1}{x}$, or equivalently $2> l_jx\geq 1$. Yet for $j\in J_\kappa$, we have $\{l_jx\}\leq \frac{1}{2}$, which together yields $l_jx \in [1,\frac{3}{2}]$. Hence $\{l_jx\}\leq \{\frac{l_jx}{2}\} \in [\frac{1}{2},\frac{3}{4}]$. Therefore, we obtain
\begin{equation}\label{eq:119}
\delta\left(\frac{x}{2}\right) = \sum_{j\in \N} \left\{l_j\frac{x}{2}\right\} \geq \sum_{j \in J_{\kappa}} \{l_jx\}\geq \sum_{j\in J_\kappa} \frac{1}{2}= \frac{1}{2}\kappa.
\end{equation}
Combining inequalities (\ref{eq:118}) and (\ref{eq:119}), and taking into account (\ref{eq:152}), we obtain
\begin{equation}\label{eq:120}
J\left(\frac{1}{x}\right)-J\left(\frac{2}{x}\right) = \sigma + \kappa \leq 2\delta(x) + 2 \delta\left(\frac{x}{2}\right){\leq} 2a_2f(x) + 2 a_2 f\left(\frac{x}{2}\right).
\end{equation}
Since $f\in \sR_{D}[\infty]$, the homogeneity property implies in particular that we can find positive constants $c_2,x_2$ (with $c_2<1$!) such that
\begin{equation*}\label{eq:154}
%c_1f(x)<
f\left(x/2\right) <c_2 f(x),
\end{equation*}
for all $x\geq x_2$. (Without loss of generality, we may assume that $x_2\geq x_0$.) Applying this to (\ref{eq:120}) %and  (\ref{eq:154})
yields
\begin{equation*}\label{eq:120-2}
J\left(\frac{1}{x}\right)-J\left(\frac{2}{x}\right)\leq 2a_2(1+c_2) f(x),
\end{equation*}
for each $x\geq x_2$, proving the assertion \eqref{eq:121} for the constant $c:=2a_2(1+c_2)$.

Now we use \eqref{eq:121}, to derive an upper bound for $J$. For $x\geq 2x_2$ let $m = m(x)$ be the unique integer such that $2^m x_2< x \leq 2^{m+1} x_2$.
%Since ${x}/{2^{m-1}}\geq 2x_0$ we may use inequality (\ref{eq:121}) $m$ times with $x$ replaced by $\frac{x}{2^{k}}$ ($k \in \{ 0,1,2...,m-1  \}$) on the following telescope sum.
Note that, by the choice of $m$, we have ${2^m}/{x}\geq {1}/{2x_2}$, which implies $J({2^m}/{x})\leq J({1}/{2x_2})=:j_0$. Writing $J(1/x)$ as a telescope sum and applying \eqref{eq:121}, we infer
%and the function $f$ is strictly increasing in $[x_0,\infty)$:
\begin{align*}%\label{eq:122}
J(1/x)
& = \sum_{k=0}^{m-1} \left( J(2^{k}/x) - J(2^{k-1}/x)  \right) + J(2^m/x)
\leq\sum_{k=0}^{m-1} c f\left({x}/{2^{k}}\right) + j_0\\
&\leq c \sum_{k=0}^{m-1}  c_2^k f(x) + j_0
\leq c f(x) \sum_{k=0}^{\infty}  c_2^k + j_0
= \frac{c}{1-c_2} f(x) + j_0,
\end{align*}
for any $x\geq 2x_2$, where the convergence of the geometric series is ensured, since $c_2 <1$. This gives an upper bound on $J(1/x)$ in terms of $f(x)$ and completes the proof. (Note that $f(x)\to\infty$ as $x\to\infty$.)
\end{proof}

To complete the proof of Theorem~\ref{theo:main}, we recall that there is a direct connection between the packing defect $\delta$ and the string counting function $N$ valid for any bounded open set $\Omega\subset\R$ independent of any additional assumptions on the growth behavior. It is given by the equation
\begin{align} \label{eq:N-delta}
   \varphi(\lambda)-N(\lambda)=\delta(\sqrt{\lambda}/\pi), \quad \lambda >0,
\end{align}
cf.\ e.g.\ \cite[eq.\ (2.2)]{LapPo1} and see also Remark~\ref{rem:N-delta} below.
 It implies the equivalence (iv) $\Leftrightarrow$ (v) in part I of Theorem~\ref{theo:main} as well as the equivalence of assertion \eqref{eq:delta2} in Theorem~\ref{theo:J-delta} with assertion \eqref{eq:main-N} in part II of Theorem~\ref{theo:main}. Therefore, by Theorem~\ref{theo:J-delta}, the assertions (vi)-(viii) imply indeed \eqref{eq:main-N}. This completes the proof.

\begin{rem} \label{rem:N-delta}
\emph{Equation \eqref{eq:N-delta} is seen as follows.
  For an open interval $I=(a,b)$ of length $l=b-a$, consider the Laplace operator $—d^2/dy^2$. Under Dirichlet boundary conditions (i.e. $u(a)=u(b)=0$) the eigenvalues are $\lambda_k=(\pi/l)^2 k^2$, $k\in\N$. Hence, for the eigenvalue counting function $N(I;\lambda)$ of $I$, we have
$$
N(I;\lambda)=\#\{k\in\N: \lambda_k\leq\lambda\}=\#\{k\in\N:  k\leq l\sqrt{\lambda}/\pi\}=[l\sqrt{\lambda}/\pi].
$$
Hence, we get for the eigenvalue counting function $N$ of  $\Omega$ (consisting of the disjoint open intervals $I_j$ of lengths $l_j$, $j\in\N$)
$$
N(\lambda)=\sum_{j=1}^\infty N(I_j;\lambda)=\sum_{j=1}^\infty [l_j \sqrt{\lambda}/\pi]=\sum_{j=1}^\infty [l_j x],
$$
where $x:=\sqrt{\lambda}/\pi$. Since $|\Omega|_1=\sum_{j=1}^\infty l_j$, it follows in particular that
$$
\varphi(\lambda)-N(\lambda)=\sqrt{\lambda}/\pi \sum_{j=1}^\infty l_j-\sum_{j=1}^\infty [l_j x]=\sum_{j=1}^\infty l_jx-\sum_{j=1}^\infty [l_j x]=\delta(x).
$$
}
\end{rem}

\begin{proof}[Summary of the proof of Theorem~\ref{theo:main}] \label{diagrams}
   The following diagram gives an overview over the various steps of the proof of part I of Theorem~\ref{theo:main}. It also shows the central role of the S-content (or the string counting function $J$) for the proof.
 
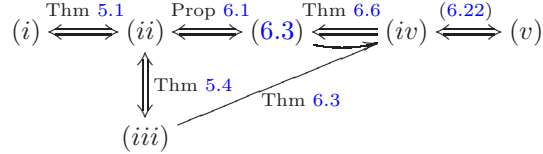
\begin{align*}
    \begin{xy}
  \xymatrix{
      (i)\ar@{<=>}[r]^{^{\text{\scriptsize Thm~\ref{thm:M-S-equiv}}}}
      & (ii) \ar@{<=>}[r]^{^{\text{\scriptsize Prop~\ref{prop:J-asymp}}}}
      & \eqref{eq:J-asymp} \ar@{<=}[r]^{^{\text{\scriptsize Thm~\ref{theo:delta-J}}}} \ar@/_/[r]
      &(iv) \ar@{<=>}[r]^{^\text{\scriptsize \eqref{eq:N-delta}}} &(v)\\
                  &   (iii)   \ar@{<=>}[u]_{\text{\scriptsize Thm~\ref{theo:4}}}   \ar@{-}[rru]_{\text{ Thm~\ref{theo:8}}}  &    &
  }
\end{xy}
 \end{align*}
	
Our proof of part II of Theorem~\ref{theo:main} has a similar structure:
 
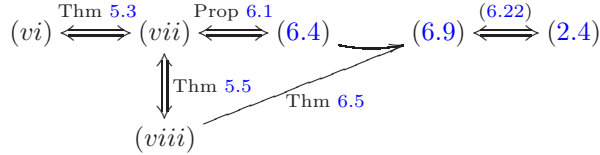
\begin{align*}
    \begin{xy}
  \xymatrix{
      (vi)\ar@{<=>}[r]^{^{\text{\scriptsize Thm~\ref{theo:11}}}}
      & (vii) \ar@{<=>}[r]^{^{\text{\scriptsize Prop~\ref{prop:J-asymp}}}}
      & \eqref{eq:J-sim}  \ar@/_/[r]
      &\eqref{eq:delta2} \ar@{<=>}[r]^{^\text{\scriptsize \eqref{eq:N-delta}}} &\eqref{eq:main-N}\\
                  &   (viii)   \ar@{<=>}[u]_{\text{\scriptsize Thm~\ref{theo:13}}}  \ar@{-}[rru]_{\text{ Thm~\ref{theo:J-delta}}}  &    &
  }
\end{xy}
 \end{align*}
	Relation \eqref{eq:main-M-S-L} between the contents and $L$ is obvious from the definition of the constant $S$ in Theorem \ref{theo:13} (and Theorem~\ref{theo:13}).
	%\begin{equation*}
%	{\cal{M}}(h;F) = {\cal{S}}(h';F) =S= \frac{2^{1-D}L^{D}}{1-D}.
%	\end{equation*}
\end{proof}

\paragraph{\bf Acknowledgements}
%This research was partially supported by the German Science Foundation (DFG), project no.\ WI 3264/2-2.
We are grateful to G\"unter Last for pointing us to the theory of regularly varying functions.
Part of the results are based on the master's thesis of the first author.

\bibliographystyle{abbrv}%amsalpha}%math}
%\nociteconference{*}
\bibliography{regvar-pub}
\end{document}